\documentclass[jsl,10pt]{noasl}
\usepackage{rotate}

\def\confname{}
\def\copyrightyear{2009}

\title[Weihrauch Degrees and Omniscience Principles]
      {Weihrauch Degrees, Omniscience Principles\\ and Weak Computability}

\author{Vasco Brattka}
\author{Guido Gherardi}

\address{Laboratory of Foundational Aspects of Computer Science\\ Department of Mathematics \& Applied Mathematics\\
University of Cape Town\\ Rondebosch 7701, South Africa} 
\address{Dipartimento di Filosofia\\ Universit\`a di Bologna\\ Italy}

\email{Vasco.Brattka@uct.ac.za}
\email{Guido.Gherardi@unibo.it}

\urladdr{http://cca-net.de/}

\thanks{This project has been supported by the Italian Ministero degli Affari Esteri and the National Research Foundation of South Africa (NRF)}

\def\AA{{\mathcal A}}
\def\BB{{\mathcal B}}
\def\CC{{\mathcal C}}

\def\FF{{\mathcal F}}
\def\GG{{\mathcal G}}

\def\KK{{\mathcal K}}

\def\OO{{\mathcal O}}

\def\RR{{\mathcal R}}

\def\WW{{\mathcal W}}

\def\IN{{\mathbb{N}}}

\def\IR{{\mathbb{R}}}

\def\IT{{\mathbb{T}}}

\def\Bar{\overline}

\def\TO{\Longrightarrow}
\def\In{\subseteq}

\def\dmin{-^{\!\!\!\!\cdot}\;}

\def\into{\hookrightarrow}

\def\prefix{\sqsubseteq}

\def\mto{\rightrightarrows}

\def\id{{\rm id}}

\def\dom{{\rm dom}}
\def\range{{\rm range}}
\def\graph{{\rm graph}}

\def\span{{\rm span}}

\def\Tr{{\rm Tr}}

\newcommand{\SO}[1]{{{\boldsymbol\Sigma}^0_{#1}}}

\def\LPO{\text{\rm\sffamily LPO}}
\def\LLPO{\text{\rm\sffamily LLPO}}
\def\WKL{\text{\rm\sffamily WKL}}

\def\HBT{\text{\rm\sffamily HBT}}

\def\BF{\text{\rm\sffamily B$_{\rm\mathsf F}$}}
\def\C{\text{\rm\sffamily C}}

\def\LPO{\text{\rm\sffamily LPO}}
\def\LLPO{\text{\rm\sffamily LLPO}}

\def\leqT{\mathop{\leq_{\mathrm{T}}}}
\def\leqM{\mathop{\leq_{\mathrm{M}}}}

\def\leqW{\mathop{\leq_{\mathrm{W}}}}
\def\equivW{\mathop{\equiv_{\mathrm{W}}}}
\def\nequivW{\mathop{\not\equiv_{\mathrm{W}}}}
\def\leqSW{\mathop{\leq_{\mathrm{sW}}}}
\def\equivSW{\mathop{\equiv_{\mathrm{sW}}}}
\def\leqPW{\mathop{\leq_{\widehat{\mathrm{W}}}}}
\def\equivPW{\mathop{\equiv_{\widehat{\mathrm{W}}}}}
\def\nleqW{\mathop{\not\leq_{\mathrm{W}}}}

\def\lW{\mathop{<_{\mathrm{W}}}}

\def\nW{\mathop{|_{\mathrm{W}}}}

\def\bigtimes{\mathop{\mathsf{X}}}
\def\topW{\mathbf{0}}
\def\botW{\mathbf{1}}
\def\emptyW{\boldsymbol{\emptyset}}

\newtheorem{proposition}{Proposition}[section]
\newtheorem{theorem}[proposition]{Theorem}
\newtheorem{lemma}[proposition]{Lemma}
\newtheorem{definition}[proposition]{Definition}
\newtheorem{corollary}[proposition]{Corollary}

\begin{document}

\begin{abstract} 
In this paper we study a reducibility that has been introduced by Klaus Weihrauch or, 
more precisely, a natural extension of this reducibility for multi-valued functions 
on represented spaces.
We call the corresponding equivalence classes Weihrauch degrees and we show that 
the corresponding partial order induces a lower semi-lattice with the disjoint union of multi-valued 
functions as greatest lower bound operation.
We prove that parallelization is a closure operator for this semi-lattice
and the parallelized Weihrauch degrees even form a lattice with the
product of multi-valued functions as greatest lower bound operation.
We show that the Medvedev lattice can be embedded into the parallelized
Weihrauch lattice in a natural way, even into the sublattice of total continuous multi-valued
functions on Baire space and such that greatest lower bounds and least upper
bounds are preserved.
As a consequence we obtain that Turing degrees can be embedded
into the single-valued part of this sublattice.
The importance of Weihrauch degrees is based on the fact that multi-valued functions on 
represented spaces can be considered as realizers of 
mathematical theorems in a very natural way and studying the Weihrauch reductions between 
theorems in this sense means to ask which theorems can be 
transformed continuously or computably into each other. 
This allows a new purely topological or computational approach to 
metamathematics that sheds new light on the nature of theorems.
As crucial corner points of this classification scheme we study the 
limited principle of omniscience $\LPO$,
the lesser limited principle of omniscience $\LLPO$ and their parallelizations.
We recall that the parallelized version of $\LPO$ is complete for limit
computable functions (which are exactly the effectively $\SO{2}$--measurable
functions in the Borel hierarchy).
We prove that parallelized $\LLPO$
is equivalent to Weak K\H{o}nig's Lemma and hence to the Hahn-Banach Theorem
in this new and very strong sense.
We call a multi-valued function weakly computable
if it is reducible to the Weihrauch degree of parallelized $\LLPO$ and we
present a new proof that the class of weakly computable operations
is closed under composition. This proof is based on a computational
version of Kleene's ternary logic. Moreover, we characterize weakly
computable operations on computable metric spaces as operations that
admit upper semi-computable compact-valued selectors and we prove
that any single-valued weakly computable operation is already computable
in the ordinary sense.
\end{abstract}



\maketitle

\section{Introduction}

The purpose of this paper is to propose a new computational approach to metamathematics
that is based on the classification of mathematical theorems according to their computational content. 
Such an approach started with a classification of the Weihrauch degree of the 
Hahn-Banach Theorem in \cite{GM09} and the intention here
is to lay some careful foundations for further studies.
In a following paper \cite{BG09b} we analyze certain choice principles and 
we present a case study with a classification on many theorems from analysis.

Essentially, the idea is to ask which theorems can be continuously or even computably transferred
into each other. In order to give a meaningful interpretation to this idea
we consider mathematical theorems as multi-valued operations $F:X\mto Y$ that 
map certain input data $X$ into certain output data $Y$. Such a perspective
is very natural, since many theorems in mathematics are actually $\Pi_2$ theorems,
hence they have the logical form
\[(\forall x\in X)(\exists y\in Y)\;(x,y)\in A\]
and one can just consider $F:X\mto Y$ as a realizer or multi-valued Skolem
function for this statement. 

The appropriate technical tool to study whether two such potentially partial multi-valued functions 
$F:\In X\mto Y$ and $G:\In X\mto Y$ can be continuously or computably transferred
into each other is Weihrauch reducibility.
This is a reducibility that has been introduced by Klaus Weihrauch around
1990 in two unpublished papers \cite{Wei92a,Wei92c} and since then
it has been studied by several others (see for instance \cite{Ste89,Myl92,Her96,Bra99,Bra05,Myl06,GM09,BG09b,Pau09}).

Originally, this reducibility has been introduced for single-valued functions on Baire space.
Basically, the idea is to say that $F$ is {\em strongly Weihrauch reducible} to $G$, 
in symbols $F\leqSW G$, if there
are computable (or alternatively continuous) functions $H$ and $K$ such that
\[F=H\circ G\circ K.\]
Thus, $K$ acts as an input modification and $H$ acts as an output modification.
We will mainly consider the computable version of this reduction here since
the positive reduction results are stronger. For negative results the topological version
of the reduction is stronger and indeed reductions typically fail for continuity reasons.
However, such topological results can usually be derived from computational results by relativization.

It turns out that the strong version of Weihrauch reducibility is slightly
too strong for many purposes, since it distinguishes too many functions. 
For instance the identity cannot be reduced to a constant function in this
way, since there is no way to feed the input through a constant function.
This is the reason why the more important reducibility is the one where we
say that $F$ is {\em Weihrauch reducible} to $G$, in symbols $F\leqW G$,
if there are computable functions $H$ and $K$ such that
\[F=H\circ\langle\id,G\circ K\rangle.\]
Here and in the following $\langle\;\rangle$ denotes suitable finite
or infinite tupling functions.
Thus, the difference is that the input is fed through to the outer function $H$
independently of $G$. Another perspective to look at this reducibility is
to say that it is the cylindrification of strong Weihrauch reducibility,
a fact which we prove in Section~\ref{sec:product-sum}.
In this sense strong and ordinary Weihrauch reducibility are related to each
other as one-one and many-one reducibility in classical computability theory,
where a similar characterization with cylindrification is well-known 
(see \cite{Soa87,Odi89} for classical computability theory).

Weihrauch \cite{Wei92a,Wei92c} has already studied an extended version of his reducibility to sets $\FF$ and $\GG$
of functions on Baire space and $\FF$ is called {\em Weihrauch reducible} to $\GG$,
in symbols $\FF\leqW\GG$, if there are computable functions $H$ and $K$ such that
\[(\forall G\in\GG)(\exists F\in\FF)\;F=H\langle\id,GK\rangle.\]
That is, any function $G\in\GG$ computes some function $F\in\FF$ and
the computation is performed uniformly with two fixed computable $H$ and $K$.
This extension of Weihrauch reducibility is related to ordinary
Weihrauch reducibility exactly as Medvedev reducibility is related
to Turing reducibility. 

We use this concept to extend Weihrauch reducibility even further to 
multi-valued operations $f:\In X\mto Y$ on represented spaces $X$ and $Y$.
Roughly speaking, such an $f$ is {\em Weihrauch reducible} to an analogous $g$,
in symbols $f\leqW g$, if the set of realizers of $f$ is reducible to the set of 
realizers of $g$ in the above mentioned sense of Weihrauch reducibility for sets, i.e.\
\[\{F:F\vdash f\}\leqW\{G:G\vdash g\}.\]
Here a single-valued $F$ on Baire space is called a {\em realizer} of $f$, in symbols $F\vdash f$,
if $F$ computes a name $F(p)$ of some output value in $f(x)$, given some name $p$ of $x$.
This generalization of Weihrauch reducibility was introduced for single-valued functions in \cite{Bra05}
and for multi-valued functions in \cite{GM09}. We call the corresponding equivalence classes
{\em Weihrauch degrees}. 

Compared to strong Weihrauch reducibility, the ordinary version of Weih\-rauch
reducibility has exactly the right degree of precision, it distinguishes exactly
what should be distinguished computationally, but not more.
Among all functions (with at least one computable point in the domain)
the computable ones form the least degree, i.e.\ they are reducible to all other such functions.
For the continuous version of Weihrauch reducibility exactly the continuous
functions form the least degree (among all functions with non-empty domain).

We study some basic properties of Weihrauch reducibility and of Weihrauch degrees in Section~\ref{sec:reducibility}.
In Section~\ref{sec:product-sum} we investigate the product operation $f\times g$
and the direct sum $f\oplus g$ of multi-valued operations and we show that both
operations are monotone with respect to Weihrauch reducibility.
While the product preserves single-valuedness, the disjoint union does not and
hence it requires multi-valuedness to be meaningful.
Among other things we prove in Section~\ref{sec:product-sum} that the partial
order on Weihrauch degrees induces a lower semi-lattice with direct sums as
greatest lower bounds.

An important operation on multi-valued functions is {\em parallelization} $\widehat{f}$ that we study in 
Section~\ref{sec:parallelization}
and it just means to take countably many copies of the function $f$ in parallel, i.e.\
\[\widehat{f}(x_0,x_1,x_2,...):=f(x_0)\times f(x_1)\times f(x_2)\times...\]
If $f$ is defined on Baire space $\IN^\IN$, then we sometimes compose parallelization
with an infinite tupling function. This is convenient, but does not affect the operation in any essential way.
We prove that this operation forms a closure operator with respect to Weihrauch
reducibility and we get a natural parallelized version of Weihrauch
reducibility.
The parallelized Weihrauch degrees together with their partial order even
form a lattice with the product as least upper bound operation\footnote{Independently, 
Arno Pauly \cite{Pau09} has recently proved 
that another operation on functions that takes direct sums on the input and
output side yields a supremum even in the non-parallelized case. He has also proved
that the corresponding upper semi-lattice is distributive.}.

In Section~\ref{sec:embedding} we prove that the Medvedev lattice can be embedded
into the Weihrauch lattice such that least upper bounds and greatest lower bounds
are preserved. This embedding only requires total and continuous multi-valued operations
on Baire space. As a consequence, we obtain that Turing degrees can be embedded such that
least upper bounds are preserved and this embedding only requires total and continuous
single-valued functions on Baire space.

In Section~\ref{sec:omniscience} we start to study the 
the {\em limited principle of omniscience} $\LPO$
and the {\em lesser limited principle of omniscience} $\LLPO$ in the 
upper semi-lattice of Weihrauch reducibility. Such a study has also
already been initiated by Weihrauch \cite{Wei92c}. The principles
themselves have originally been introduced by Brouwer and Bishop
in constructive mathematics \cite{BB85,BR87a}.
Roughly speaking, $\LPO$ corresponds to the law of the excluded middle $(A\vee\neg A)$ 
and $\LLPO$ to de Morgan's law $\neg (A\wedge B)\iff(\neg A\vee\neg B)$, both
restricted to simple existential statements. More precisely, they are stated as 
follows:

\begin{itemize}
\item (\LPO) For any sequence $p\in\IN^\IN$ there exists an $n\in\IN$ such that $p(n)=0$ or $p(n)\not=0$ for all $n\in\IN$.
\item (\LLPO) For any sequence $p\in\IN^\IN$ such that $p(k)\not=0$ for at most one $k\in\IN$, 
it follows $p(2n)=0$ for all $n\in\IN$ or $p(2n+1)=0$ for all $n\in\IN$.
\end{itemize}

If we interpret them as mathematical theorems in the above mentioned
sense, then we can consider $\LPO:\IN^\IN\to\IN$ and $\LLPO:\In\IN^\IN\mto\IN$
as operations and characterize their Weihrauch degree.
In Section~\ref{sec:LLPO} we show that $\LPO$ is strictly Weihrauch reducible to 
$\LLPO$ (which was also already proved by Weihrauch) and the parallelized
version $\widehat{\LPO}$ of $\LPO$ is computably complete for the Borel class
of effectively $\SO{2}$--measurable operations \cite{Bra05}. This completely characterizes
the Weihrauch degree of $\widehat{\LPO}$. 

In Section~\ref{sec:LLPO} we prove that also the parallelized version $\widehat{\LLPO}$ of $\LLPO$
has some very interesting properties. Firstly, $\widehat{\LLPO}$ is closed
under composition and secondly we use a computable version of Kleene's 
ternary logic to show that the whole cone of multi-valued operations under $\widehat{\LLPO}$
is closed under composition (a result which was already proved in a different way in \cite{GM09}).

This justifies to give a special name to the operations $F:\In X\mto Y$ below $\widehat{\LLPO}$
and we call them {\em weakly computable}. In Section~\ref{sec:WKL}
we show that the Weihrauch
degree of Weak K\H{o}nig's Lemma is the same as that of $\widehat{\LLPO}$
and hence it follows from results in \cite{GM09} that it is identical
to the Weihrauch degree of the Hahn-Banach Theorem. 
We also prove that compact choice has the same Weihrauch degree and we
derive from this result that the weakly computable operations on computable
metric spaces $X,Y$ are exactly those that admit an upper semi-computable
compact-valued selector. Finally, we prove that this implies that
any single-valued weakly computable operation is already computable in the ordinary
sense.

This has surprising algorithmic consequences. Any ``algorithm'' that uses weakly computable
operations such as 
$x\leq0\mbox{ or }x\geq 0$
leads to a uniformly computable result, as long as the result is uniquely determined, i.e.\
single-valued. And this is so, although these operations are typically discontinuous
and non-computable.

In the Conclusions we discuss the relevance of the observations made in this
paper. In particular, we claim that the classification of the Weihrauch degree
of mathematical theorems sheds some new light on these theorems. 
The study that has been started in this paper is continued in \cite{BG09b}.

\section{Weihrauch reducibility of multi-valued operations}
\label{sec:reducibility}

In this section we define Weihrauch reducibility for multi-valued
functions on represented spaces and we study some basic properties of it.
In a first step we define the concept for sets of functions on Baire space,
as it was already considered by Weihrauch \cite{Wei92a,Wei92c}.

\begin{definition}[Weihrauch reducibility]\rm
Let $\mathcal F$ and $\mathcal G$ be sets of functions of type $F:\In\IN^\IN\to\IN^\IN$.
We say that $\FF$ is {\em Weihrauch reducible} to $\GG$,
in symbols $\mathcal F\leqW\mathcal G$, if there are computable 
functions $H,K:\In\IN^\IN\to\IN^\IN$ such that
\[(\forall G\in\mathcal G)(\exists F\in\mathcal F)\;F=H\langle\id,GK\rangle.\]
Analogously, we define $\FF\leqSW\GG$ using the equation $F=HGK$ and in this
case we say that $\FF$ is {\em strongly Weihrauch reducible} to $\GG$.
\end{definition}

Here $\langle\;\rangle:\IN^\IN\times\IN^\IN\to\IN^\IN$ denotes
a computable standard pairing function \cite{Wei00}.
That is, this reducibility is derived from Weihrauch reducibility of functions
in the same way as Medvedev reducibility is derived from Turing reducibility
in classical computability theory \cite{Rog67}.
We denote the induced equivalence relations by $\equivW$ and $\equivSW$, respectively.

In the next step we define the concept of a realizer of a multi-valued function
as it is used in computable analysis \cite{Wei00}. We recall that a {\em representation}
$\delta_X:\In\IN^\IN\to X$ of a set $X$ is a surjective (and potentially partial) map.
In this situation we say that $(X,\delta_X)$ is a {\em represented space}.

\begin{definition}[Realizer]\rm
Let $(X,\delta_X)$ and $(Y,\delta_Y)$ be represented spaces
and let $f:\In X\mto Y$ be a multi-valued function.
Then $F:\In\IN^\IN\to\IN^\IN$ is called a {\em realizer} of $f$ with respect to
$(\delta_X,\delta_Y)$, in symbols $F\vdash f$, if 
\[\delta_YF(p)\in f\delta_X(p)\]
for all $p\in\dom(f\delta_X)$.
\end{definition}

Usually, we do not mention the representations explicitly 
since they will be clear from the context.
A multi-valued function $f:\In X\mto Y$ on represented spaces
is called {\em continuous} or {\em computable}, if it has a continuous
or computable realizer, respectively.
Using reducibility for sets and the concept of a realizer we can now
define Weihrauch reducibility for multi-valued functions.

\begin{definition}[Realizer reducibility]\rm
Let $f$ and $g$ be multi-valued functions on represented spaces. 
Then $f$ is said to be {\em Weihrauch reducible}
to $g$, in symbols $f\leqW g$, if and only if $\{F:F\vdash f\}\leqW\{G:G\vdash g\}$.
Analogously, we define $f\leqSW g$ with the help of $\leqSW$ on sets.
\end{definition}

That is, $f\leqW g$ holds if any realizer of $g$ computes some realizer of $f$
with some fixed uniform translations $H$ and $K$.
This reducibility has already been used in \cite{Bra05} for single-valued
maps and in \cite{GM09} for multi-valued maps.
We mention that 
we also write $f\lW g$ if and only if $f\leqW g$
and $g\nleqW f$. Moreover, we write $f\nW g$ if $f\nleqW g$
and $g\nleqW f$. Analogous notations are used for $\leqSW$.
It is clear that Weihrauch reducibility and its strong version form 
preorders, i.e.\ both relations are reflexive and transitive.
We use standard computable tupling functions $\pi:\IN^\IN\times\IN^\IN\to\IN^\IN$
and $\pi':(\IN^\IN)^\IN\to\IN^\IN$ and the function values are denoted
by $\langle p,q\rangle:=\pi(p,q)$ and $\langle p_0,p_1,p_2,...\rangle:=\pi'(p_i)_{i\in\IN}$.
By $\pi_i:\IN^\IN\to\IN^\IN$ we denote the computable projections given by
$\pi_1\langle p,q\rangle:=p$ and $\pi_2\langle p,q\rangle=q$.
For functions $F,G:\In\IN^\IN\to\IN^\IN$ we define $F\otimes G:\In\IN^\IN\to\IN^\IN$
by $(F\otimes G)\langle p,q\rangle:=\langle F(p),G(q)\rangle$ for all $p,q\in\IN^\IN$
and $\langle F,G\rangle:\In\IN^\IN\to\IN^\IN$ by 
$\langle F,G\rangle(p):=\langle F(p),G(p)\rangle$ for all $p\in\IN^\IN$.

\begin{lemma}[Preorders]
\label{lem:preorder}
The relations $\leqW$ and $\leqSW$ on multi-valued functions on represented spaces 
are both preorders, i.e.\ they are reflexive and transitive.
Hence $\equivW$ and $\equivSW$ are equivalence relations.
\end{lemma}
\begin{proof}
If one chooses $H=\pi_2$
and $K=\id$, then one obtains $H\langle\id,GK\rangle=G$ for any function
$G:\In\IN^\IN\to\IN^\IN$. Hence $\leqW$ is reflexive on sets and hence on multi-valued
operations.

Now let $e:\In W\mto Z$, $f:\In X\mto Y$ and $g:\In U\mto V$ be 
multi-valued operations of represented spaces. If $e\leqW f$ and $f\leqW g$ holds,
then there are computable functions $H,K,H',K':\In\IN^\IN\to\IN^\IN$
such that $H\langle\id,GK\rangle$ is a realizer of $f$ for any realizer $G$ of $g$
and $H'\langle\id,FK'\rangle$ is a realizer of $e$ for any realizer $F$ of $f$.
Then $H'':=H'\langle\pi_1,H(K'\otimes\id)\rangle$ and $K'':=KK'$ are both computable and 
\[
H''\langle\id,GK''\rangle
= H'\langle\id,H(K'\otimes\id)\langle\id,GKK'\rangle\rangle
= H'\langle\id,H\langle\id,GK\rangle K'\rangle
\]
is a realizer of $e$ for any realizer $G$ of $g$. Thus, $e\leqW g$.
Analogously, one can show that $\leqSW$ is reflexive and transitive.
\end{proof}

We did not specify exactly what the domain of our relations $\leqW$
and $\leqSW$ is. If one chooses the class of all multi-valued operations
on represented spaces, then one does not obtain a set. 
This is the reason why we assume from now on that we have some given
set $\RR$ of represented spaces and we consider all multi-valued operations
between them. 

\begin{definition}[Weihrauch degree]\rm
A {\em Weihrauch degree} is an equivalence classes with respect to 
$\equivW$ of all multi-valued operations $f:\In X\mto Y$ on
represented spaces $X,Y\in\RR$. 
\end{definition}

It is clear that $\leqW$ induces a partial order on Weihrauch degrees.
It is a straightforward observation that strong Weihrauch reducibility
is actually stronger than the ordinary one and both reducibilities 
preserve continuity and computability.

\begin{proposition}
\label{prop:reducibility-computability}
Let $f$ and $g$ be multi-valued functions on represented spaces. Then 
\begin{enumerate}
\item $f\leqSW g\TO f\leqW g$,
\item $f\leqW g$ and $g$ computable $\TO$ $f$ computable,
\item $f\leqW g$ and $g$ continuous $\TO$ $f$ continuous.
\end{enumerate}
\end{proposition}

We leave the straightforward proofs to the reader.
Another observation is that the nowhere defined functions $g:\In X\mto Y$ form
the least Weihrauch degree. This is because any nowhere defined function $g$
has a realizer that is nowhere defined and hence any $f\leqW g$ must also be nowhere
defined. On the other hand, any nowhere defined function is clearly reducible
to any other function.

\begin{lemma}[Nowhere defined function]
\label{lem:nowhere}
The nowhere defined multi-val\-ued functions $f:\In X\mto Y$ form the
least Weihrauch degree.
\end{lemma}

Despite this observation we will typically exclude the nowhere defined functions
from our considerations.
We believe that the weaker version $\leqW$ of Weihrauch reducibility
is more useful and, in fact, more natural than the stronger version. 
This is mainly because of the following observation. We recall that 
a point $x\in X$ in a represented space $(X,\delta)$ is called {\em computable},
if there is a computable $p\in\IN^\IN$ such that $\delta(p)=x$.

\begin{lemma}
\label{least:least}
Let $f$ and $g$ be multi-valued functions 
on represented spaces. If $f$ is computable and $\dom(g)$ contains
a computable point, then $f\leqW g$.
\end{lemma}
\begin{proof}
Let $q$ be a computable name of a point in $\dom(g)$. Let us consider a computable realizer
$F:\In\IN^\IN\to\IN^\IN$ of $f$. Then $H:=F\pi_1$ and $K$ with $K(p):=q$ for all $p\in\IN^\IN$
are computable and $H\langle\id,GK\rangle=F$ is a realizer of $f$ for any realizer
$G$ of $g$. Hence $f\leqW g$.
\end{proof}

We obtain as an immediate corollary that among those functions 
with at least one computable point in the domain the computable ones
form the least Weihrauch degree.

\begin{corollary}[Least degree]
\label{cor:least}
Among all multi-valued functions with at least 
one computable point in the domain, the computable ones form the least
Weihrauch degree.
\end{corollary}

The functions without a computable point in their domain are not very
interesting for most practical purposes. For all multi-valued operations
on represented spaces that are of practical interest, the domain itself
is a represented space with computable points.
For strong Weihrauch reducibility $\leqSW$ the previous result does not
hold true. It is easy to see that for instance the identity cannot 
be strongly reduced to any constant function.
Thus, strong Weihrauch reducibility distinguishes between functions
that we want to consider as essentially equivalent.

We mention that for both reducibilities, Weihrauch reducibility and strong
Weihrauch reducibility, there is a continuous counterpart where the 
reduction functions $H$ and $K$ are replaced by continuous functions.
Some of our results hold analogously for continuous Weihrauch reducibility.
Since we do not want to introduce further symbols for continuous reducibility,
we will typically express such results by saying that a reduction holds
``with respect to some oracle''. That is, we exploit the fact that a function
is continuous if and only if it is computable with respect to some oracle.
Analogously to the previous corollary we obtain the following topological
version.

\begin{corollary}
Among all multi-valued functions with non-empty domain the continuous 
ones form the least Weihrauch degree for Weihrauch reducibility with 
respect to some oracle. 
\end{corollary}

We now want to show that our concept of reducibility is 
invariant under equivalent representations. 
Such invariance properties are of particular importance in computable analysis \cite{Wei00}.
If we have two representations
$\delta$ and $\delta'$ of a set $X$, then $\delta$ is said to be {\em reducible}
to $\delta'$, in symbols $\delta\leq\delta'$,
if there is a computable function $F:\In\IN^\IN\to\IN^\IN$ such that
$\delta(p)=\delta'F(p)$ for all $p\in\dom(\delta)$. 
Moreover, $\delta$ is said to be {\em equivalent} to $\delta'$, in
symbols $\delta\equiv\delta'$, if $\delta\leq\delta'$ and $\delta'\leq\delta$.
We can now formulate the following result.

\begin{lemma}[Invariance under representations]
Let $f$ and $g$ be multi-valued functions.
If $f\leqW g$ holds with respect to certain representations
and each representation is replaced by an equivalent one, then $f\leqW g$ also
holds with respect to these equivalent representations.
An analogous statement holds for $\leqSW$.
\end{lemma}
\begin{proof}
We consider functions $f:\In X\mto Y$ and $g:\In U\mto V$ and
let $\delta_X,\delta_X',\delta_Y,\delta_Y',\delta_U,\delta_U'$ and $\delta_V,\delta_V'$
be representations of $X,Y,U$ and $V$, respectively. Moreover, let $Q,R,S,T:\In\IN^\IN\to\IN^\IN$
be computable functions such that $\delta_X'=\delta_XQ$, $\delta_Y=\delta_Y'R$,
$\delta_U=\delta_U'S$ and $\delta_V'=\delta_VT$. 
Now we assume that $f\leqW g$ holds with respect to $\delta_X,\delta_Y,\delta_U,\delta_V$.
That is, there are computable functions $H,K:\In\IN^\IN\to\IN^\IN$ such that
\[\delta_YH\langle\id,GK\rangle\in f\delta_X(p)\]
for all $p\in\dom(f\delta_X)$ and for all functions $G:\In\IN^\IN\to\IN^\IN$ 
that are realizers of $g$ with respect to $(\delta_U,\delta_V)$, i.e.\ for which
$\delta_VG(p)\in g\delta_U(p)$ for all $p\in\dom(\delta_VG)$.
Then also the functions $H',K':\In\IN^\IN\to\IN^\IN$ with 
$H'=RH(Q\otimes T)$ and $K':=SKQ$ are computable.
Let $G':\In\IN^\IN\to\IN^\IN$ be a realizer of $g$ with respect to $(\delta_U',\delta_V')$,
i.e.\ $g$ satisfies
$\delta_V'G'(p)\in g\delta_U'(p)$ for all $p\in\dom(g\delta_U')$. 
Then $\delta_VTG'S(p)\in g\delta_U(p)$ for all $p\in\dom(g\delta_U)$,
i.e.\ $G:=TG'S$ is a realizer of $g$ with respect to $(\delta_V,\delta_U)$
and hence we obtain
\begin{eqnarray*}
\delta_Y'H'\langle p,G'K'(p)\rangle
&=& \delta_Y'RH(Q\otimes T)\langle p,G'SKQ(p)\rangle\\
&=& \delta_YH\langle Q(p),TG'SKQ(p)\rangle\\
&=& \delta_YH\langle\id,GK\rangle(Q(p))\\
&\in& f\delta_XQ(p)\\
&=& f\delta_X'(p)
\end{eqnarray*}
for all $p\in\dom(f\delta_X')$.
Thus $f\leqW g$ with respect to $\delta_X',\delta_Y',\delta_U',\delta_V'$.
The statement for $\leqSW$ can be proved analogously.
\end{proof}

\section{Product and sum of Weihrauch degrees}
\label{sec:product-sum}

Now we want to study the product and disjoint union operation of multi-valued
operations. They are related to the supremum and infimum for
Weihrauch reducibility. We start with the product operation.

\begin{definition}[Product]\rm
Let $f:\In X\mto Y$ and $g:\In U\mto V$ be multi-valued functions on represented spaces. 
Then the {\em product} of these maps $f\times g:\In X\times U\mto Y\times V$ is defined by 
\[(f\times g)(x,u):=f(x)\times g(u)\]
for all $(x,u)\in\dom(f\times g)=\dom(f)\times\dom(g)$.
\end{definition}

We assume that whenever $(X,\delta_X)$ and $(U,\delta_U)$ are represented spaces,
then the product $X\times U$ is represented by the canonical product
representation $[\delta_X,\delta_U]$, defined by $[\delta_X,\delta_U]\langle p,q\rangle:=(\delta_X(p),\delta_U(q))$.
We prove that the product is a monotone operation with
respect to Weihrauch reducibility.

\begin{proposition}[Monotonicity of products]
\label{prop:monotone}
Let $f,f',g$ and $g'$ be mul\-ti-valued functions on represented spaces. Then
\[f\leqW g\mbox{ and }f'\leqW g'\TO f\times f'\leqW g\times g'.\]
An analogous statement holds for strong Weihrauch reducibility $\leqSW$.
\end{proposition}
\begin{proof}
We consider maps $f:\In X\mto Y$, $f':\In Z\mto W$, $g:\In U\mto V$ and $g':\In S\mto T$
on represented spaces $(X,\delta_X)$, $(Y,\delta_Y)$, $(Z,\delta_Z)$, $(W,\delta_W)$, $(U,\delta_U)$, $(V,\delta_V)$,
$(S,\delta_S)$ and $(T,\delta_T)$.
Let $H,H',K,K':\In\IN^\IN\to\IN^\IN$ be computable functions such that
$H\langle\id,GK\rangle$ is a realizer of $f$ for any realizer $G$ of $g$ 
and such that $H'\langle\id,G'K'\rangle$ is a realizer of $f'$ for any realizer $G'$ of $g'$.
We use the projections $\pi_i$ to define
$P:=\langle\langle\pi_1\pi_1,\pi_1\pi_2\rangle,\langle\pi_2\pi_1,\pi_2\pi_2\rangle\rangle$
and we define computable functions $H'':=(H\otimes H')P$ and $K'':=(K\otimes K')$.
Now let $G''$ be a realizer of $g\times g'$ with respect to the product representation, i.e.\
\[[\delta_V,\delta_T]G''\langle u,s\rangle\in (g\times g')[\delta_U,\delta_S]\langle u,s\rangle\]
for all $\langle u,s\rangle\in\dom((g\times g')[\delta_U,\delta_S])$.
We fix a pair $\langle p,q\rangle\in\dom((f\times f')[\delta_X,\delta_Z])$.
Then there are realizers $G$ of $g$ and $G'$ of $g'$ such that $G''\langle K(p),K'(q)\rangle=\langle GK(p),G'K'(q)\rangle$.
We obtain
\begin{eqnarray*}
H''\langle\id,G''K''\rangle\langle p,q\rangle
&=& (H\otimes H')P\langle\id,G''(K\otimes K')\rangle\langle p,q\rangle\\
&=& (H\otimes H')P\langle\langle p,q\rangle,\langle GK(p),G'K'(q)\rangle\rangle\\
&=& (H\otimes H')\langle\langle p,GK(p)\rangle,\langle q,G'K'(q)\rangle\rangle\\
&=& \langle H\langle p,GK(p)\rangle,H'\langle q,G'K'(q)\rangle\rangle,
\end{eqnarray*}
and hence $H''\langle\id,G''K''\rangle$ is a realizer of $f\times f'$. 
This shows $f\times f'\leqW g\times g'$.
The result for strong Weihrauch reducibility can be proved analogously.
\end{proof}

This monotonicity result guarantees that we can safely extend the product
operation to Weihrauch degrees.
Since $f\leqW f\times g$ and $g\leqW f\times g$ (given that $f$ and $g$ have at least one computable
point in the domain), it follows that $f\times g$ is a common upper bound of $f$ and $g$.
Often it will also be the least upper bound. However, this is not always the case
since there are functions $f$, even single-valued ones, which are not idempotent.

\begin{lemma}
There are functions $f:\In X\to Y$ such that $f\nequivW f\times f$. 
\end{lemma}

We will provide a concrete example in Corollary~\ref{cor:LPO-LLPO-idempotency}.
Such a function $f$ necessarily has to be discontinuous, since all
computable functions (with at least one computable point in the domain)
are equivalent and, in particular, idempotent. In general, we call
a Weihrauch degree {\em idempotent}, if $f\equivW f\times f$ holds
for some $f$ in that degree (and hence for all $f$ in that degree by Proposition~\ref{prop:monotone}).

Using products we can characterize the relation between strong
and ordinary Weihrauch reducibility. In fact, it can be expressed
in similar terms as the relation between many-one reducibility and one-one
reducibility using cylindrifications.

\begin{definition}[Cylindrification]\rm
Let $f:\In X\mto Y$ be a function on represented spaces.
We call $\id\times f$ with $\id:\IN^\IN\to\IN^\IN$ the
{\em cylindrification} of $f$ and we call $f$ a {\em cylinder},
if $f\equivSW\id\times f$.
\end{definition}

If not mentioned otherwise, we assume that the identity is defined
on Baire space $\IN^\IN$. We also assume that $\IN^\IN$ is represented
by the identity,
which is equivalent to the Cauchy representation of $\IN^\IN$. In particular, any single-valued function
on Baire space can be considered as its own realizer.
Now we can prove the following result on the relation of ordinary
Weihrauch reducibility and strong reducibility. Roughly speaking this result shows
that reduction between two functions is equivalent to strong reduction between their
cylindrifications.

\begin{proposition}[Cylindrification]
\label{prop:cylinder}
For all multi-valued functions $f$ and $g$ on represented spaces we obtain
\[f\leqW g\iff\id\times f\leqSW\id\times g.\]
\end{proposition}
\begin{proof}
Let us assume that $f\leqW g$ holds, i.e.\ there are computable functions
$H$ and $K$ such that $H\langle\id,GK\rangle$ is a realizer of $f$ for any realizer
$G$ of $g$. Then $H',K'$, defined by 
\[H'\langle\langle p,q\rangle,r\rangle=\langle p,H\langle q,r\rangle\rangle
 \mbox{ and }
 K'\langle p,q\rangle=\langle\langle p,q\rangle, K(q)\rangle\]
are computable.
Let $G'$ be a realizer of $\id\times g$. 
We fix a pair $\langle p,q\rangle$ such that $q$ is a name of a point in $\dom(f)$.
Then there is a realizer $G$ of $g$ such that 
$G'\langle\langle p,q\rangle,K(q)\rangle=\langle\langle p,q\rangle,GK(q)\rangle$
and we obtain
\begin{eqnarray*}
H'G'K'\langle p,q\rangle
&=& H'G'\langle\langle p,q\rangle,K(q)\rangle\\
&=& H'\langle\langle p,q\rangle,GK(q)\rangle\\
&=& \langle p,H\langle q,GK(q)\rangle\rangle
\end{eqnarray*}
and $H'G'K'$ is a realizer of $\id\times f$, 
which proves $(\id\times f)\leqSW(\id\times g)$.

Now let $(\id\times f)\leqSW(\id\times g)$. Then there are computable
functions $H,K:\In\IN^\IN\to\IN^\IN$ such that $HIK$ is a realizer of $\id\times f$
for any realizer $I$ of $\id\times g$. In this situation $\pi_2HIKD$ is a realizer of $f$,
where $D(p)=\langle p,p\rangle$.
We define computable functions $H':=\pi_2H(\pi_1KD\otimes\id)$ and $K':=\pi_2KD$.
Now let $G$ be a realizer of $g$.
Then $I:=\id\otimes G$ is a realizer of $\id\times g$ and we obtain:
\[
H'\langle\id,GK'\rangle(p)
= \pi_2H(\pi_1KD\otimes\id)\langle\id,G\pi_2KD\rangle(p)
= \pi_2HIKD(p),
\]
i.e.\ $H'\langle\id,GK'\rangle$ is a realizer of $f$ whenever $G$ is a realizer of $g$.
This shows, $f\leqW g$. 
\end{proof}

It is also easy to see that cylindrification is a closure operator
on strong Weihrauch degrees and the cylindrification of strong 
Weihrauch degrees just yields the ordinary Weihrauch degrees.
We do not discuss this any further here. We just formulate
a corollary that shows that Weihrauch reducibility and strong
reducibility to cylinders are identical.

\begin{corollary}[Reductions to cylinders]
\label{cor:cylinder-reduction}
Let $f$ and $g$ be multi-valued functions on represented space
and let $g$ be a cylinder.
Then we obtain 
\[f\leqW g\iff f\leqSW g.\]
\end{corollary}

This is a consequence of the fact that $f\leqSW\id\times f$ always holds.
In the next proposition we collect
a number of algebraic properties of the product operation.
In particular, it turns out that Weihrauch degrees form a commutative
monoid with respect to the products.

\begin{proposition}[Product]
\label{prop:product}
Let $f,g$ and $h$ be multi-valued functions on represented spaces. Then
\begin{enumerate}
\item $(f\times g)\times h\equivSW f\times(g\times h)$ \hfill (associative)
\item $f\times g\equivSW g\times f$ \hfill (commutative)
\item $f\times\id\equivW\id\times f\equivW f$ \hfill (identity)
\end{enumerate}
\end{proposition}

We leave the straightforward proofs to the reader, one just has to use
tupling functions and projections appropriately.
We can say that strong Weihrauch degrees form a semi-group with
respect to the product $\times$, whereas ordinary Weihrauch degrees
are even a monoid with the degree of the identity $\id:\IN^\IN\to\IN^\IN$ (i.e.\
the degree of computable functions) as neutral element.
In this sense the usage of ordinary Weihrauch reducibility 
opposed to strong Weihrauch reducibility can also be motivated algebraically.
As a next operation we want to discuss the direct sum of multi-valued maps.
For any two sets $Y,Z$ we define the {\it direct sum} or {\it disjoint union} by
$Y\oplus Z:=(\{0\}\times Y)\cup(\{1\}\times Z)$.

\begin{definition}[Direct sum]\rm
Let $f:\In X\mto Y$ and $g:\In U\mto V$ be multi-valued maps on represented spaces.
Then the {\em direct sum} of these maps $f\oplus g:\In X\times U\mto Y\oplus V$ is defined by
\[(f\oplus g)(x,u):=(\{0\}\times f(x))\cup(\{1\}\times g(u))\]
for all $(x,u)\in\dom(f\oplus g):=\dom(f)\times\dom(g)$.
\end{definition}

If $(Y,\delta_Y)$ and $(V,\delta_V)$ are represented spaces, then the
direct sum $Y\oplus V$ is represented by $\delta_Y\sqcup\delta_V$, defined 
by 
\[(\delta_Y\sqcup\delta_V)(np):=\left\{\begin{array}{ll}
  \{0\}\times\delta_Y(p) & \mbox{if $n=0$}\\
  \{1\}\times\delta_V(p) & \mbox{otherwise}
\end{array}\right.\]
for all $n\in\IN$, $p\in\IN^\IN$.

One should note that in contrast to the product operation the direct sum
operation does not preserve single-valuedness. 
Thus, the direct sum operation requires multi-valuedness in order to be meaningful.
A nice property of the direct sum operation is that it gives us the
greatest lower bound with respect to Weihrauch reducibility.
We first prove that the direct sum operation is strongly idempotent.
Here and in the following we will occasionally use the computable {\em left shift operation}
$L:\IN^\IN\to\IN^\IN$, defined by $L(p)(n):=p(n+1)$ for all $p\in\IN^\IN$ and $n\in\IN$.

\begin{proposition}[Idempotency]
\label{prop:idempotent-sum}
Let $f$ be a multi-valued map on represented spaces. 
Then we obtain $f\equivSW f\oplus f$.
\end{proposition}
\begin{proof}
Let $G$ be a realizer of $f\oplus f$. 
Then by 
\[F(p):=LG\langle p,p\rangle\]
we get a realizer of $f$. It is clear that this shows $f\leqSW f\oplus f$.
Now let $F$ be an arbitrary realizer of $f$. Then by
\[G\langle p,q\rangle:=0F(p)\]
we obtain a realizer of $f\oplus f$. This shows $f\oplus f\leqSW f$.
\end{proof}

Now we prove a monotonicity result for sums analogously to the result for products in 
Proposition~\ref{prop:monotone}.

\begin{proposition}[Monotonicity of sums]
\label{prop:monotone-sum}
Let $f,f',g,g'$ be multi-valued functions on represented spaces. Then
\[f\leqW g\mbox{ and }f'\leqW g'\TO f\oplus f'\leqW g\oplus g'.\]
An analogous statement holds for strong Weihrauch reducibility $\leqSW$.
\end{proposition}
\begin{proof}
Let $H,H',K,K':\In\IN^\IN\to\IN^\IN$ be computable functions such that
$F=H\langle\id,GK\rangle$ is a realizer of $f$ for any realizer $G$ of $g$
and $F'=H'\langle\id,G'K'\rangle$ is a realizer of $f'$ for any realizer $G'$ of $g'$.
Define $K'':=(K\otimes K')$ and $H''$ by 
\[H''\langle\langle p,q\rangle,nr\rangle:=\left\{\begin{array}{ll}
  0H\langle p,r\rangle  & \mbox{if $n=0$}\\
  1H'\langle q,r\rangle & \mbox{otherwise}
\end{array}\right..\]
Then $H''$ and $K''$ are computable.
Let $G''$ be a realizer of $g\oplus g'$.
We fix a name $\langle p,q\rangle$ of an element in $\dom(f\oplus f')$.
Then there are realizers $G$ of $g$ and $G'$ of $g'$ such that
\[LG''\langle K(p),K'(q)\rangle=\left\{\begin{array}{ll}
  GK(p) & \mbox{if $G''\langle K(p),K'(q)\rangle(0)=0$}\\
  G'K'(q) & \mbox{otherwise}
\end{array}\right..\]
We obtain
\begin{eqnarray*}
&&  H''\langle\id,G''K''\rangle\langle p,q\rangle\\
&=& H''\langle\langle p,q\rangle,G''\langle K(p),K'(q)\rangle\rangle\\
&=& \left\{\begin{array}{ll}
    0H\langle p,LG''\langle K(p),K'(q)\rangle\rangle & \mbox{if $G''\langle K(p),K'(q)\rangle(0)=0$}\\
    1H'\langle q,LG''\langle K(p),K'(q)\rangle\rangle & \mbox{otherwise}
    \end{array}\right.\\
&=& \left\{\begin{array}{ll}
    0H\langle p,GK(p)\rangle & \mbox{if $G''K''\langle p,q\rangle(0)=0$}\\
    1H'\langle q,G'K'(q)\rangle & \mbox{otherwise}
    \end{array}\right..\\
\end{eqnarray*}
Thus $H''\langle\id,G''K''\rangle$ is a realizer of $f\oplus f'$. 
It follows that $f\oplus f'\leqW g\oplus g'$.
The result for strong Weihrauch reducibility can be proved analogously.
\end{proof}

This result shows, in particular, that the direct sum operation $\oplus$ can
be straightforwardly extended to Weihrauch degrees of multi-valued functions.
And more than this, they form a lower semi-lattice with the direct sum operation as
greatest lower bound operation.

\begin{proposition}[Greatest lower bound]
\label{prop:greatest-lower-bound}
Let $f$ and $g$ be multi-valued functions on represented spaces.
Then $f\oplus g$ is the greatest lower bound of $f$ and $g$
with respect to Weihrauch reducibility $\leqW$ and strong
Weihrauch reducibility $\leqSW$.
\end{proposition}
\begin{proof}
If $h$ is a common lower bound of $f$ and $g$, i.e.\ $h\leqW f$ and $h\leqW g$,
then $h\oplus h\leqW f\oplus g$ by Proposition~\ref{prop:monotone-sum}.
But $h\equivW h\oplus h$ by Proposition~\ref{prop:idempotent-sum}.
This implies $h\leqW f\oplus g$.
On the other hand, it is easy to see that $f\oplus g\leqW f$ and $f\oplus g\leqW g$ hold.
If, for instance, $F$ is a realizer of $f$, then by $G\langle p,q\rangle:=0F(p)$ a realizer
of $f\oplus g$ is obtained.
The statement for strong reducibility can be proved analogously.
\end{proof}

We collect the algebraic properties of the sum operation in the following proposition.

\begin{proposition}[Sum]
\label{prop:sum}
Let $f,g$ and $h$ be multi-valued functions on represented spaces. Then
\begin{enumerate}
\item $f\equivSW f\oplus f$ \hfill (idempotent)
\item $(f\oplus g)\oplus h\equivSW f\oplus(g\oplus h)$ \hfill (associative)
\item $f\oplus g\equivSW g\oplus f$ \hfill (commutative)
\end{enumerate}
\end{proposition}

Is there any multi-valued map that plays the role of a neutral element 
with respect to the sum operation?
Naturally, this would have to be a multi-valued function with an empty set of realizers.
One should note that this is not the nowhere defined function $f:\In X\mto Y$,
since $\{F:F\vdash f\}$ is the set of all function $F:\In\IN^\IN\to\IN^\IN$.
If we accept the Axiom of Choice, then clearly, a function without realizers does not exist and hence we add an extra
object $\emptyW$ to our structure with $\{F:F\vdash\emptyW\}=\emptyset$. 
Weihrauch reducibility can straightforwardly be extended to multi-valued functions enriched
by $\emptyW$, just by using $\emptyset$ as the set of realizers of $\emptyW$.
We denote the Weihrauch degree of $\emptyW$ by $\topW$.
Once again we assume that we have a fixed underlying set of represented spaces $\RR$
and now we also assume that this set includes $(\IN^\IN,\id)$ and that $\RR$ is closed
under products and direct sums.

\begin{definition}[Set of Weihrauch degrees]\rm
Let $\WW$ denote the set that contains the degree $\topW$ and all Weihrauch degrees 
of all multi-valued operations $f:\In X\mto Y$ with at least one computable point in $\dom(f)$
and with represented spaces $X,Y\in\RR$.
By $\botW$ we denote the degree of the computable functions in $\WW$.
\end{definition}

In the following theorem we collect all the structural properties of Weih\-rauch degrees that we have
studied so far.

\begin{theorem}[Weihrauch degrees]
\label{thm:Weihrauch}
The space $(\WW,\leqW)$ of Weihrauch degrees is a lower semi-lattice with least element $\botW$
and greatest element $\topW$ and with $\oplus$ as the greatest lower bound operation.
In particular, $(\WW,\oplus)$ is an idempotent monoid with neutral element $\topW$.
Moreover, $(\WW,\times)$ is a monoid with neutral element $\botW$.
\end{theorem}

We note that all the results regarding the product $\times$ also hold true if we
restrict the consideration to single-valued functions on represented spaces,
or even more concrete, to the set of single-valued functions $f:\In\IN^\IN\to\IN^\IN$
on Baire space (with at least one computable point in the domain). 
This is because the product $\times$ preserves single-valuedness, in 
contrast to the sum $\oplus$.
We also mention that the underlying set $\RR$ of represented spaces can always be assumed to be
some Cartesian closed category of admissibly represented spaces, 
whenever that is useful \cite{Sch02}.

\section{Parallelization of Weihrauch degrees}
\label{sec:parallelization}

In this section we study parallelization and we show that it is
a closure operator on Weihrauch degrees.
Parallelization can be considered as infinite product of an operation
with itself. We also show that parallelized Weihrauch degrees form a
lattice.

\begin{definition}[Parallelization]\rm
Let $f:\In X\mto Y$ be a multi-valued function. Then we define
the {\em parallelization} $\widehat{f}:\In X^\IN\mto Y^\IN$ of $f$ by 
\[\widehat{f}(x_i)_{i\in\IN}:=\bigtimes_{i=0}^\infty f(x_i)\]
for all $(x_i)_{i\in\IN}\in X^\IN$.
\end{definition}

We also assume that whenever a set $X$ is represented by $\delta:\In\IN^\IN\to X$, then
the sequence set $X^\IN$ is represented by $\delta^\IN:\In\IN^\IN\to X^\IN$, defined by
\[\delta^\IN\langle p_0,p_1,p_2,...\rangle:=(\delta(p_i))_{i\in\IN}.\]
Consequently, it follows that whenever $F$ is a realizer of $f$, then
$\overline{F}$ is a realizer of $\widehat{f}$, where 
\[\overline{F}:\In\IN^\IN\to\IN^\IN,\langle p_0,p_1,p_2,...\rangle\mapsto\langle F(p_0),F(p_1),F(p_2),...\rangle.\]
It is clear that for single-valued functions $F$ 
on Baire space $\widehat{F}\equivSW\overline{F}$.
We prove that parallelization acts as a closure operator with
respect to Weihrauch reducibility.

\begin{proposition}[Parallelization]
\label{prop:parallelization}
Let $f$ and $g$ be multi-valued functions on represented spaces. Then
\begin{enumerate}
\item $f\leqW\widehat{f}$                       \hfill (extensive)
\item $f\leqW g\TO\widehat{f}\leqW\widehat{g}$  \hfill (increasing)
\item $\widehat{f}\equivW\;\widehat{\!\!\widehat{f}}$ \hfill (idempotent)
\end{enumerate}
An analogous result holds for strong Weihrauch reducibility.
\end{proposition}
\begin{proof}
We consider $f:\In X\mto Y$ and $g:\In U\mto V$.
If $F:\In\IN^\IN\to\IN^\IN$ is a realizer of $\widehat{f}:\In X^\IN\to Y^\IN$,
then this realizer can be used to compute a realizer $F'$ of $f$ by
\[F'(p)=\pi_0F\langle p,p,p,...\rangle.\]
This shows $f\leqW\widehat{f}$. Here 
$\pi_0:\IN^\IN\to\IN^\IN,\langle p_0,p_1,p_2,...\rangle\mapsto p_0$
denotes the projection on the first component.
Now, if $f\leqW g$, then there are computable $H$ and $K$, such
that for any realizer $G$ of $g$, the function $H\langle\id,GK\rangle$
is a realizer for $f$. Then we define a computable function $L$ by
\[L\langle\langle p_0,p_1,p_2,...\rangle,\langle q_0,q_1,q_2,...\rangle\rangle:=
  \langle\langle p_0,q_0\rangle,\langle p_1,q_1\rangle,\langle p_2,q_2\rangle,...\rangle\]
and we obtain that whenever $G'$ is a realizer of $\widehat{g}$, then
$\overline{H}L\langle\id,G'\overline{K}\rangle$ is a realizer of $\widehat{f}$.
Finally, if $F$ is a realizer for $\widehat{f}$, then the function $F'$,
defined by
\[F'\langle\langle p_{\langle 0,0\rangle},p_{\langle 0,1\rangle},p_{\langle 0,2\rangle},...\rangle,\langle p_{\langle 1,0\rangle},p_{\langle 1,1\rangle},p_{\langle 1,2\rangle},...\rangle,...\rangle
  :=HF\langle p_0,p_1,p_2,...\rangle\]
with
\[H\langle q_0,q_1,q_2,...\rangle:=\langle\langle q_{\langle 0,0\rangle},q_{\langle 0,1\rangle},q_{\langle 0,2\rangle},...\rangle,\langle q_{\langle 1,0\rangle},q_{\langle 1,1\rangle},q_{\langle 1,2\rangle},...\rangle,...\rangle\]
is a realizer of $\widehat{\widehat{f}}$.
Essentially the same proof also shows that parallelization acts as a closure operator
for strong Weihrauch reducibility.
\end{proof}

The fact that Weihrauch reducibility is a closure operator allows
us to define a parallelized version of Weihrauch reducibility.

\begin{definition}[Parallel reducibility]\rm
Let $f$ and $g$ be multi-valued operations on represented spaces.
Then we say that $f$ is {\em parallely Weihrauch reducible} to $g$, 
in symbols $f\leqPW g$, if $\widehat{f}\leqW\widehat{g}$.
We say that $f$ is {\em parallely Weihrauch equivalent} to $g$, in symbols
$f\equivPW g$, if $f\leqPW g$ and $g\leqPW f$ holds.
We call the corresponding equivalence classes {\em parallel Weihrauch degrees}.
\end{definition}

It is cleat that $\leqPW$ is a preorder, i.e.\ it is reflexive and transitive,
since it inherits these properties from $\leqW$ (see Lemma~\ref{lem:preorder}).
The fact that parallelization is a closure operator gives us the following
alternative way of characterizing parallel Weihrauch reducibility.

\begin{lemma}
Let $f$ and $g$ be multi-valued operations on represented spaces.
Then
\[f\leqPW g\iff f\leqW\widehat{g}.\]
\end{lemma}

We call a multi-valued function $f$ on represented spaces {\em parallelizable}
if $f\equivW\widehat{f}$. Correspondingly, we call a Weihrauch degree {\em parallelizable},
if it has a parallelizable member. This terminology is similar to cylindrification.
Obviously, as a consequence of the previous result we obtain that
for parallelizable $g$ we have $f\leqW g$ if and only if $f\leqPW g$.
Parallel Weihrauch degrees have somewhat nicer algebraic features
than Weihrauch degrees. This is essentially, because parallelization
commutes with products in the following sense and because parallel degrees are idempotent.

\begin{proposition}[Products and parallelization]
\label{prop:product-parallelization}
Let $f$ and $g$ be multi-valued operations on represented spaces. Then
\[\widehat{f\times g}\equivSW\widehat{f}\times\widehat{g}.\]
\end{proposition}
\begin{proof}
 If $f$ and $g$
are of type $f:\In X\mto Y$ and $g:\In U\mto V$, then $\widehat{f\times g}:\In (X\times U)^\IN\mto (Y\times V)^\IN$ and 
$\widehat{f}\times\widehat{g}:\In X^\IN\times U^\IN\mto Y^\IN\times V^\IN$.
We can identify these two operations using the computable homeomorphism
\[h_{X,U}:X^\IN\times U^\IN\to (X\times U)^\IN,((x_i)_{i\in\IN},(u_i)_{i\in\IN})\mapsto(x_i,u_i)_{i\in\IN}\]
for $X,U$ and an analogous map $h_{Y,V}$ for $Y\times V$.
More specifically, we obtain
\[\widehat{f\times g}=h_{Y,V}(\widehat{f}\times\widehat{g})h_{X,U}^{-1}\mbox{ and }
  (\widehat{f}\times\widehat{g})=h_{Y,V}^{-1}(\widehat{f\times g})h_{X,U}.
\]
This proves the claim since the homeomorphisms and their inverses are computable.
\end{proof}

This result allows us to consider the product as operation on parallel
Weihrauch degrees, because it implies, in particular, that the product
of two parallel degrees is parallel again. 
Similarly, we can prove the following result.

\begin{proposition}[Idempotency and parallelization]
\label{prop:absorb}
For a multi-valued function $f$ on represented spaces we obtain
\[\widehat{f}\equivSW\widehat{f}\times\widehat{f}.\]
\end{proposition}
\begin{proof}
We prove $\widehat{f}\times\widehat{f}\leqSW\widehat{f}$. If $f$ is of type
$f:\In X\mto Y$, then $\widehat{f}:\In X^\IN\mto Y^\IN$ and 
$\widehat{f}\times\widehat{f}:\In X^\IN\times X^\IN\mto Y^\IN\times Y^\IN$.
We can identify these two operations using the computable homeomorphism
\[h:X^\IN\times X^\IN\to X^\IN\mbox{ with }h((x_i)_{i\in\IN},(x_i')_{i\in\IN})(n):=\left\{\begin{array}{ll}
  x_i & \mbox{if $n=2i$}\\
  x_i' & \mbox{if $n=2i+1$}
\end{array}\right.\]
for $X$ and an analogous map for $Y$.
This proves $\widehat{f}\times\widehat{f}\leqSW\widehat{f}$. 
It is clear that $\widehat{f}\leqSW\widehat{f}\times\widehat{f}$ holds.
\end{proof}

We can also formulate the following version of this observation.

\begin{corollary}
Any parallelizable Weihrauch degree is idempotent.
\end{corollary}

The idempotency of parallel Weihrauch degrees has the consequence that the product 
actually is the least upper bound operation for parallel Weihrauch degrees.

\begin{proposition}[Least upper bound]
\label{prop:least-lower-bound}
Let $f$ and $g$ be multi-valued functions on represented spaces.
Then $f\times g$ is the least upper bound of $f$ and $g$
with respect to parallel Weihrauch reducibility $\leqPW$.
\end{proposition}
\begin{proof}
If $h$ is a common parallel upper bound of $f$ and $g$, i.e.\ $f\leqW\widehat{h}$ and $g\leqW\widehat{h}$,
then $f\times g\leqW\widehat{h}\times\widehat{h}$ by Proposition~\ref{prop:monotone}.
But $\widehat{h}\equivW\widehat{h}\times\widehat{h}$ by Proposition~\ref{prop:absorb}.
This implies $f\times g\leqW\widehat{h}$.
On the other hand, it is easy to see that $f\leqW f\times g$ and $g\leqW f\times g$ hold,
if $f$ and $g$ have at least one computable point in their domain.
\end{proof}

The next result treats the interaction of parallelization with sums.
We note that we do not prove that parallelization commutes with sums,
but just that the sum of parallelized functions is parallelizable.

\begin{proposition}[Sums and parallelization]
\label{prop:sum-parallelization}
Let $f$ and $g$ be multi-valued operations on represented spaces. Then
\[\widehat{f\oplus g}\leqSW\widehat{\widehat{f}\oplus\widehat{g}}\equivSW\widehat{f}\oplus\widehat{g}.\]
\end{proposition}
\begin{proof}
If $f$ and $g$ are of type $f:\In X\mto Y$ and $g:\In U\mto V$, then 
$\widehat{f}\oplus\widehat{g}:\In X^\IN\times U^\IN\mto Y^\IN\oplus V^\IN$ and
$\widehat{\widehat{f}\oplus\widehat{g}}:\In(X^\IN\times U^\IN)^\IN\mto(Y^\IN\oplus V^\IN)^\IN$.
We define two computable maps $h':(X^\IN\times U^\IN)^\IN\to X^\IN\times U^\IN$ by
\[h'((x_{ij})_{i\in\IN},(u_{ij})_{i\in\IN})_{j\in\IN}:=
  ((x_{ij})_{\langle i,j\rangle\in\IN},(u_{ij})_{\langle i,j\rangle\in\IN}).\]
and $h:Y^\IN\oplus V^\IN\to(Y^\IN\oplus V^\IN)^\IN$ by
\[h(k,(z_{\langle i,j\rangle})_{\langle i,j\rangle\in\IN}):=(k,(z_{\langle i,j\rangle})_{i\in\IN})_{j\in\IN}\]
for all $(k,(z_{\langle i,j\rangle})_{\langle i,j\rangle\in\IN})\in Y^\IN\oplus V^\IN$.
Then we obtain
\[h(\widehat{f}\oplus\widehat{g})h'(x,u)\In(\widehat{\widehat{f}\oplus\widehat{g}})(x,u)\]
for all $(x,u)\in(X^\IN\times U^\IN)^\IN$. Since $h$ and $h'$ are computable,
this implies that $\widehat{\widehat{f}\oplus\widehat{g}}\leqSW\widehat{f}\oplus\widehat{g}$.
It is clear that $\widehat{f}\oplus\widehat{g}\leqSW\widehat{\widehat{f}\oplus\widehat{g}}$
since parallelization is a closure operator and $\widehat{f\oplus g}\leqSW\widehat{\widehat{f}\oplus\widehat{g}}$
holds since $f\leqSW\widehat{f}$ and $g\leqSW\widehat{g}$ and by monotonicity of sums
according to Proposition~\ref{prop:monotone-sum}
\end{proof}

This result shows that the sum operation on parallelized Weihrauch degree
is well-defined. However, we cannot define this operation using arbitrary
representatives of a parallel Weihrauch degree, we have to use a parallelized
representative.

By $\widehat{\WW}$ we denote the set of parallel Weihrauch
degrees, which is defined as $\WW$ but using parallel Weihrauch reducibility.
As a corollary of our results we obtain that the parallel Weihrauch degrees 
of multi-valued functions form a lattice.

\begin{theorem}[Parallel Weihrauch degrees]
\label{thm:Weihrauch-parallel}
The space $(\widehat{\WW},\leqPW)$ of parallel Weihrauch degrees is a lattice with least element $\botW$
and greatest element $\topW$, with $\oplus$ as the greatest lower bound operation
and with $\times$ as the least upper bound operation.
In particular, $(\widehat{\WW},\oplus)$ and $(\widehat{\WW},\times)$ are idempotent monoids 
with neutral elements $\topW$ and $\botW$, respectively.
\end{theorem}

At the end of this section we mention that one should not have any wrong
expectations on how products and parallelization interact.
The parallelization of a function is not necessarily the supremum
of all its finite products with itself (see Corollary~\ref{cor:LPO-product}).
This also indicates that we can capture significantly finer distinctions
with Weihrauch degrees than with parallelized Weihrauch degrees.

\section{Embedding of Turing degrees and Medvedev degrees}
\label{sec:embedding}

In this section we want to prove that Turing degrees and Medvedev degrees
can be embedded into parallelized Weihrauch degrees quite naturally.
In fact, it is sufficient to embed Medvedev degrees, since the embedding
of Turing degrees is a special case.

We recall that a set $\AA\In\IN^\IN$ is said to be {\it Medvedev reducible} to $\BB\In\IN^\IN$, in symbols
$\AA\leq_{\rm M}\BB$, if there 
exists a computable $F:\In\IN^\IN\to\IN^\IN$ with $\BB\In\dom(F)$ and $F(\BB)\In\AA$.
In fact, Turing reducibility is a special case, since $p\in\IN^\IN$ is said to be
{\em Turing reducible} to $q\in\IN^\IN$, in symbols $p\leqT q$, if $\{p\}\leq_{\rm M}\{q\}$
(see \cite{Rog67}). 

Now we associate to any $q\in\IN^\IN$ the {\em constant function} 
\[c_q:\IN^\IN\to\IN^\IN, p\mapsto q\]
for all $p\in\IN^\IN$. 
In the next step we associate a multi-valued function to any non-empty $\AA\In\IN^\IN$ by
\[c_\AA:\IN^\IN\mto\IN^\IN,p\mapsto\AA\]
for all $p\in\IN^\IN$. Then $c_\AA$ has a computable realizer if and only
if $\AA$ contains a computable member.
To the empty set $\emptyset\In\IN^\IN$ we associate $c_\emptyset:=\emptyW$, the
special ``multi-valued function'' without realizer.
We note that the function $c_\AA$ is parallelizable, i.e.\ $c_\AA\equivW\widehat{c_\AA}$.
Our main result of this section is now the following theorem.

\begin{theorem}[Embedding of Medvedev degrees]
\label{thm:Medvedev}
Let $\AA,\BB\In\IN^\IN$. Then
\[\AA\leqM\BB\iff c_\AA\leqW c_\BB.\]
\end{theorem}
\begin{proof}
Without loss of generality, we assume that $\AA$ and $\BB$ are non-empty,
since the result obviously holds otherwise.
Let us assume that $c_\AA\leqW c_\BB$. Then there exist computable
functions $H,K:\In\IN^\IN\to\IN^\IN$ such that 
$H\langle\id,GK\rangle$ is a realizer of $c_\AA$ for any
realizer $G$ of $c_\BB$. Being a realizer of $c_\AA$ means that
$H\langle\id,GK\rangle(\widehat{0})\in\AA$, where $\widehat{0}$ is the constant zero sequence.
If $p\in\BB$, then $c_p$ is a realizer of $c_\BB$. We define $F:\In\IN^\IN\to\IN^\IN$ by
\[F(p):=H\langle\id,c_pK\rangle(\widehat{0})=H\langle\widehat{0},p\rangle\]
for all $p\in\IN^\IN$.
Then $F$ is computable, and if $p\in\BB$, then $F(p)\in\AA$,
i.e.\ $\BB\In\dom(F)$ and $F(\BB)\In\AA$. Thus, $\AA\leqM\BB$.

Now let us suppose that $\AA\leqM\BB$, i.e.\ there exists a computable function
$F:\In\IN^\IN\to\IN^\IN$ such that $\BB\In\dom(F)$ and $F(\BB)\In\AA$.
We have to prove $c_\AA\leqW c_\BB$, i.e.\ we have to provide computable
$H,K:\In\IN^\IN\to\IN^\IN$ such that $H\langle\id,GK\rangle$
is a realizer of $c_\AA$ for any realizer $G$ of $c_\BB$.
For this purpose we choose $K=\id$ and we define $H:\In\IN^\IN\to\IN^\IN$ by
$H\langle p,q\rangle:=F(q)$ for all $p,q\in\IN^\IN$.
If $G$ is a realizer of $c_\BB$, then $G(p)\in\BB$ for any $p\in\IN^\IN$ and we obtain
\[H\langle p,GK(p)\rangle=FG(p)\in\AA\]
for all $p\in\IN^\IN$, i.e.\ $H\langle\id,GK\rangle$ is a realizer of $c_\AA$.
This proves $c_\AA\leqW c_\BB$.
\end{proof}

It is clear that a corresponding embedding of Turing degrees follows,
i.e. $p\leqT q\iff c_p\leqW c_q$. 
Since Turing reducibility is mostly considered
for subsets $A\In\IN$, we formulate a corresponding corollary for completeness.

\begin{corollary}[Embedding of Turing degrees]
\label{cor:Turing}
Let $A,B\In\IN$. Then 
\[A\leqT B\iff c_{\graph(A)}\leqW c_{\graph(B)}.\]
\end{corollary}

%
The reader should notice that for the embedding of Medvedev degrees 
we have only used a certain fraction of the parallel Weihrauch lattice $\widehat{\WW}$,
namely only the sublattice for total and continuous multi-valued functions $f:\IN^\IN\mto\IN^\IN$ (and
$\emptyW$). It is easy to see that product $\times$ and direct sum $\oplus$ preserve
totality and continuity and in fact products, sequences and sums of Baire space $\IN^\IN$ can easily 
be identified with $\IN^\IN$. Similarly, we obtain that for the embedding of
Turing degrees we only need single-valued total and continuous functions $f:\IN^\IN\to\IN^\IN$.

Now we want to show that our embedding of the Medvedev lattice preserves
also greatest lower and least upper bounds. For sets $\AA,\BB\In\IN^\IN$ 
one usually defines 
\[\AA\oplus\BB:=\{\langle p,q\rangle:p\in\AA\mbox{ and }q\in\BB\}
\mbox{ and } 
\AA\otimes\BB:=0\AA\cup1\BB.
\]
The reader should note that product and sum are just swapped compared
to the way we use these operations. Now one can easily prove the following result.

\begin{proposition}
Let $\AA,\BB\In\IN^\IN$. Then
\[c_{\AA\oplus\BB}\equivSW c_\AA\times c_\BB\mbox{ and }c_{\AA\otimes\BB}\equivSW c_\AA\oplus c_\BB.\]
\end{proposition}
\begin{proof}
We just note that 
\[c_{\langle p,q\rangle}(r)=\langle p,q\rangle\mbox{ and }(c_p\times c_q)(r,s)=(p,q)\]
for all $r,s\in\IN^\IN$. Thus $c_{\langle p,q\rangle}\equivSW c_p\times c_q$ follows,
using the computable tupling function $\pi:\IN^\IN\times\IN^\IN\to\IN^\IN$ with 
$\pi(p,q):=\langle p,q\rangle$ and its inverse
and analogously one obtains the desired result $c_{\AA\oplus\BB}\equivSW c_\AA\times c_\BB$.

For the other equivalence, we note that
\[c_{\AA\otimes\BB}(r)=\AA\otimes\BB=0\AA\cup 1\BB\mbox{ and }(c_\AA\oplus c_\BB)(r,s)=(\{0\}\times\AA\cup\{1\}\times\BB)\]
for all $r,s\in\IN^\IN$.
One can easily see that this implies $c_{\AA\otimes\BB}\equivSW c_\AA\oplus c_\BB$.
\end{proof}

We mention that this result implies that our embedding of the Medvedev lattice
preserves least upper bounds and greatest lower bounds.

\begin{corollary}[Embedding of the Medvedev lattice]
The Medvedev lattice is embeddable into the parallel Weihrauch lattice (restricted
to total and continuous multi-valued functions on Baire space and $\emptyW$) with an embedding
that preserves least upper bounds and greatest lower bounds.
\end{corollary}

We also formulate the analogous result for Turing degrees.

\begin{corollary}[Embedding of the Turing upper semi-lattice]
The upper\linebreak semi-lattice of Turing degrees is embeddable into the parallel
Weihrauch lattice (restricted to total and continuous single-valued functions on Baire space)
with an embedding that preserves least upper bounds.
\end{corollary}

Using these results some structural properties of the parallel Weihrauch lattice
can be transferred from the Turing uppers semi-lattice and the Medvedev lattice.
This observation also gives raise to plenty of further research questions.

\section{Omniscience principles}
\label{sec:omniscience}

In this section we will study the omniscience principles
that we mentioned already in the introduction. We will consider
them in form of the following maps.

\begin{definition}[Omniscience principles]\rm
We define:
\begin{itemize}
\item $\LPO:\IN^\IN\to\IN,\hspace{4.9mm}\LPO(p)=\left\{\begin{array}{ll}
       0 & \mbox{if $(\exists n\in\IN)\;p(n)=0$}\\
       1 & \mbox{otherwise}
       \end{array}\right.$,
\item $\LLPO:\In\IN^\IN\mto\IN,\LLPO(p)\ni\left\{\begin{array}{ll}
       0 & \mbox{if $(\forall n\in\IN)\;p(2n)=0$}\\
       1 & \mbox{if $(\forall n\in\IN)\;p(2n+1)=0$}
       \end{array}\right.$,
\end{itemize}
where $\dom(\LLPO):=\{p\in\IN^\IN:p(k)\not=0$ for at most one $k\}$.
\end{definition}

One should notice that the definition of $\LLPO$ implies that $\LLPO(0^\IN)=\{0,1\}$.
The natural numbers $\IN$ can be represented
by $\delta_\IN(p):=p(0)$, but for simplicity of notation we will usually work directly with $\IN$.


The two principles $\LPO$ and $\LLPO$ have already been studied in computable analysis \cite{Wei92a,Wei92c,Ste89,Myl92}.
For instance, it is well-known that $\LPO$ is reducible to any other discontinuous single-valued function
on Baire space (see Lemma~8.2.6 in \cite{Wei00}). For completeness we include the proof.

\begin{proposition}
\label{prop:discontinuous}
Let $F:\In\IN^\IN\to\IN^\IN$ be discontinuous. Then we obtain $\LPO\leqSW F$, relatively to some oracle.
\end{proposition}
\begin{proof}
Since $\dom(F)$ is a subspace of a metric space, it is first countable and hence sequential. 
Thus, $F$ is continuous if and only if it sequentially continuous.
Let $q$ be a point of discontinuity of $F$. Then there is a sequence $(q_n)_{n\in\IN}$
in $\dom(F)$ that converges to $q$, but such that $(F(q_n))_{n\in\IN}$ does not converge
to $F(q)$. Without loss of generality, we can even assume that there is a $k$ such that
\[(\forall n)\;F(q_n)[k]\not=F(q)[k],\]
since we can select a suitable subsequence otherwise.
We consider $w:=F(q)[k]$ and the characteristic function $\chi_{w\IN^\IN}:\IN^\IN\to\IN$
of the ball $w\IN^\IN$.
Now we define a function $K:\IN^\IN\to\IN^\IN$ by
\[K(p)=\left\{\begin{array}{ll}
q_n & \mbox{if $n$ is minimal with $p(n)=0$}\\
q   & \mbox{if $p$ contains no $0$}
\end{array}\right..\]
Then $K$ is continuous and we obtain
\begin{eqnarray*}
\chi_{w\IN^\IN}FK(p)=0
&\iff& FK(p)\not\in w\IN^\IN\\
&\iff& FK(p)[k]\not=w\\
&\iff& K(p)\not=q\\
&\iff& (\exists n)\;p(n)=0\\
&\iff& \LPO(p)=0.
\end{eqnarray*}
Thus $\LPO(p)=\chi_{w\IN^\IN}FK(p)$
for all $p\in\IN^\IN$. Now $\chi_{w\IN^\IN}$ is computable and $K$ is continuous,
hence they are both computable with respect to some oracle. 
This shows $\LPO\leqSW F$, relatively to some oracle.
\end{proof}

While $\LPO$ is the ``simplest'' single-valued discontinuous function, 
its parallelization $\widehat{\LPO}$ is at the other end of the spectrum,
it is complete among all $\SO{2}$--measurable functions with respect to the
Borel hierarchy. 
For simplicity, we want to consider $\widehat{\LPO}$ as map of type $\IN^\IN\to\IN^\IN$
instead of type $(\IN^\IN)^\IN\to\IN^\IN$. That is, whenever we write $\widehat{\LPO}$
in the following, we actually mean $\widehat{\LPO}\circ\pi$ with
the tupling function $\pi:(\IN^\IN)^\IN\to\IN^\IN,\langle p_0,p_1,p_2,...\rangle\mapsto(p_0,p_1,p_2,...)$.
Since $\pi$ and $\pi^{-1}$ are computable, it is clear that $\widehat{\LPO}\equivSW\widehat{\LPO}\circ\pi$.
An analogous remark holds true for $\LLPO$.
We use the following map $\C:\IN^\IN\to\IN^\IN$, which is defined by
\[\C(p)(n):=\left\{\begin{array}{ll}
  0 & \mbox{if $(\exists k)\;p\langle n,k\rangle=0$}\\
  1 & \mbox{otherwise}
\end{array}\right.
\]
for all $p\in\IN^\IN$ and $n\in\IN$. This map has been studied already in \cite{Ste89,Myl92,Bra99,Bra05}
and it is known that it is $\SO{2}$--complete (see below).

\begin{lemma}
\label{lem:LPO-C}
$\widehat{\LPO}=\C$.
\end{lemma}
\begin{proof}
The claim follows directly from
\begin{eqnarray*}
\widehat{\LPO}\langle p_0,p_1,p_2,...\rangle(n)
&=& \left\{\begin{array}{ll}
     0 & \mbox{if $(\exists k)\;p_n(k)=0$}\\
     1 & \mbox{otherwise}
    \end{array}\right.\\
&=& \left\{\begin{array}{ll}
    0 & \mbox{if $(\exists k)\;p\langle n,k\rangle=0$}\\
    1 & \mbox{otherwise}
    \end{array}\right.\\
&=& \C(p)(n)
\end{eqnarray*}
for $p=\langle p_0,p_1,p_2,...\rangle$.
\end{proof}

The following result follows from Theorem~7.6 in \cite{Bra05}. The definitions of $\SO{2}$--computability
and $\SO{2}$--measurability can also be found in \cite{Bra05}. We assume that computable metric spaces
are represented with their Cauchy representations.

\begin{corollary}[Completeness]
\label{cor:completeness}
Let $X,Y$ be computable metric spaces and let $k\in\IN$.
For any function $f:X\to Y$ we obtain:
\begin{enumerate}
\item $f$ is $\SO{2}$--measurable $\iff f\leqW\widehat{\LPO}$ with respect to some oracle,
\item $f$ is $\SO{2}$--computable $\iff f\leqW\widehat{\LPO}$.
\end{enumerate}
\end{corollary}

This result can be generalized to higher classes of $\SO{k}$, see \cite{Bra05}.
The $\SO{2}$--computable maps are also called {\em limit computable}.
In the next section we will see that the parallelization of $\LLPO$ 
also corresponds to a very nice class of effective operations.
Here we continue to formulate some further basic observation 
about $\LPO$ and $\LLPO$.

\begin{proposition}
\label{prop:LLPO-cylinder}
The operations $\widehat{\LPO}$ and $\widehat{\LLPO}$ are cylinders.
\end{proposition}
\begin{proof}
By Proposition~\ref{prop:absorb} it is sufficient to show that
$\id:\IN^\IN\to\IN^\IN$ is strongly reducible to $\widehat{\LPO}$ and $\widehat{\LLPO}$. 
If we define $K:\IN^\IN\to\IN^\IN$ by
\[K(p)\langle\langle k,m\rangle, n\rangle:=\left\{\begin{array}{ll}
  0 & \mbox{if $p(k)=m$}\\
  1 & \mbox{otherwise}
\end{array}\right.\]
and $H:\In\IN^\IN\to\IN^\IN$ by
\[H(q)(k)=\min\{m\in\IN:q\langle k,m\rangle=0\},\]
then $H$ and $K$ are computable and we obtain
\begin{eqnarray*}
\C K(p)\langle k,m\rangle
&=& \left\{\begin{array}{ll}
    0 & \mbox{if $(\exists n)\;K(p)\langle\langle k,m\rangle, n\rangle=0$}\\
    1 & \mbox{otherwise}
    \end{array}\right.\\
&=& \left\{\begin{array}{ll}
    0 & \mbox{if $p(k)=m$}\\
    1 & \mbox{otherwise}
    \end{array}\right.
\end{eqnarray*}
and hence $H\C K(p)=p$, which proves $\id\leqSW\C=\widehat{\LPO}$ by Lemma~\ref{lem:LPO-C}.

Now we define a computable $K':\IN^\IN\to\IN^\IN$ by
\[K'(p)\langle\langle k,m\rangle, n\rangle:=\left\{\begin{array}{ll}
  1 & \mbox{if $p(k)=m$ and $n$ odd}\\
  1 & \mbox{if $p(k)\not=m$ and $n$ even}\\
  0 & \mbox{otherwise}
\end{array}\right.\]
and we use $H$ as above in order to obtain
\begin{eqnarray*}
\widehat{\LLPO}\circ K'(p)\langle k,m\rangle
&\ni& \left\{\begin{array}{ll}
    0 & \iff(\forall n)\;K'(p)\langle\langle k,m\rangle, 2n\rangle=0\\
    1 & \iff(\forall n)\;K'(p)\langle\langle k,m\rangle, 2n+1\rangle=0
    \end{array}\right.\\
&=& \left\{\begin{array}{ll}
    0 & \mbox{if $p(k)=m$}\\
    1 & \mbox{otherwise}
    \end{array}\right.\\
\end{eqnarray*}
and hence $H\circ\widehat{\LLPO}\circ K'(p)=p$, which proves $\id\leqSW\widehat{\LLPO}$.
\end{proof}

Next we prove that the operations $\LPO$ and $\LLPO$ are not idempotent.
For any function $f$ we denote by $f^{(k)}=\bigtimes_{i=1}^k f$ the $k$--fold
product of $f$ with itself. We recall that a {\em limit machine} is a Turing
machine that is allowed to revise the output and such machines can be 
used to characterize exactly the limit computable functions. A limit
machine is said to make at most $k$ mind changes, if the machine goes back
on the output tape at most $k$ many times (each time for an arbitrary finite number of cells).

\begin{proposition}
\label{prop:LPO-LLPO-mind-change}
Let $k\in\IN$.
The operations $\LPO^{(k+1)}$ and $\LLPO^{(k+1)}$ can both be computed
on a limit machine with at most $k+1$ mind changes, but not with $k$ mind changes.
\end{proposition}
\begin{proof}
We describe a limit machine for $\LPO^{(k+1)}$ that requires $k+1$
mind changes. Upon input $(p_1,...,p_{k+1})\in(\IN^\IN)^{k+1}$ the machine
writes $(1,1,...,1)\in\{0,1\}^{k+1}$ as default output and it continues
to inspect the input tuple. As soon as some $n,i$ is found with $p_i(n)=0$,
the corresponding $i$--th component of the output is changed from $1$ to $0$.
This computation requires at most $k+1$ mind changes.
On the other hand, one can easily see that no limit machine can compute
$\LPO^{(k+1)}$ with less than $k+1$ mind changes. Starting with input
$(1^\IN,1^\IN,...,1^\IN)\in(\IN^\IN)^{(k+1)}$ the machine eventually
has to produce output $(1,1,...,1)$. If this happens in time step $t$,
then one can change the first input sequence by adding a $0$ to it
in position $t+1$. This forces the machine to make a mind change and
to produce a new output $(0,1,1,...,1)$ after $t'$ time steps.
Then one changes the second input sequence and so on. Altogether,
the limit machine will have to make $k+1$ mind changes.
The fact for $\LLPO$ can be proved analogously.
\end{proof}

Since the number of mind changes required by a limit machine is invariant
under Weihrauch reducibility (see the Mind Change Principle in \cite{BG09b}),
we get the following corollary.

\begin{corollary}
\label{cor:LPO-LLPO-idempotency}
We obtain $\LPO^{(k)}\lW\LPO^{(k+1)}$ and $\LLPO^{(k)}\lW\LLPO^{(k+1)}$ for all $k\in\IN$.
In particular, $\LPO$ and $\LLPO$ are not idempotent.
\end{corollary}

Moreover, similarly as we have shown $\LPO\lW\BF$ in \cite{BG09b},
we can prove more generally $\LPO^{(k)}\lW\BF$ for all $k\in\IN$. We do not want to define
$\BF$ here, but we mention that it is easy to see that it is idempotent and that $\BF\lW\widehat{\LPO}$.
Thus, we get the following corollary.

\begin{corollary}
\label{cor:LPO-product}
There exists a single-valued function $f$ that is idempotent and satisfies
$\LPO^{(k)}=\bigtimes_{i=1}^k\LPO\lW f\lW\bigtimes_{i=1}^\infty \LPO=\widehat{\LPO}$
for all $k\in\IN$.
\end{corollary} 

Thus, the parallelization of $\LPO$ is not the supremum of the finite products of $\LPO$.
An analogous result can be proved for $\LLPO$.

\section{The lesser limited principle of omniscience and weak computability}
\label{sec:LLPO}

In this section we want to study the parallelization of $\LLPO$.
Similarly, as $\widehat{\LPO}$ is complete for the class of limit
computable operations, we will show that $\widehat{\LLPO}$ is also
complete for a very natural class of operations that we will call {\em weakly computable}.

We recall that by $\widehat{\LLPO}$ we actually mean $\widehat{\LLPO}\circ\pi$.
Thus, in the following
\[\widehat{\LLPO}\langle p_0,p_1,...\rangle(k)\ni\left\{\begin{array}{ll}
  0 & \mbox{if $(\forall n)\;p_k(2n)=0$}\\
  1 & \mbox{if $(\forall n)\;p_k(2n+1)=0$}
\end{array}\right.\]
One benefit of this understanding of $\widehat{\LLPO}$ is that it is
composable with itself and
the next observation is that the composition of $\widehat{\LLPO}$ with itself
is strongly below itself\footnote{In general, we define the composition $g\circ f$ of two multi-valued 
maps $f:\In X\mto Y$ and $g:\In Y\mto Z$ by $\dom(g\circ f):=\{x\in\dom(f):f(x)\In\dom(g)\}$
and $(g\circ f)(x):=\{z\in Z:(\exists y)(y\in f(x)\mbox{ and }z\in g(y))\}$.}.
Roughly speaking this is because $\LLPO$ is defined
only in terms of universal quantifiers and two consecutive universal quantifiers 
can be absorbed in one.

\begin{lemma} 
\label{lem:LLPO-double}
$\widehat{\LLPO}\circ\widehat{\LLPO}\leqSW\widehat{\LLPO}$.
\end{lemma}
\begin{proof}
We define a computable function $F:\IN^\IN\to\IN^\IN$ by
\[\left\{\begin{array}{lcl}
F(p)\langle k,2\langle n,m\rangle\rangle &:=& p\langle\langle k,2n\rangle,2m\rangle\\
F(p)\langle k,2\langle n,m\rangle+1\rangle &:=& p\langle\langle k,2n+1\rangle,2m\rangle
\end{array}\right.\]
for all $p\in\IN^\IN$ and $k,n,m\in\IN$. Then we obtain
\begin{eqnarray*}
\widehat{\LLPO}\circ\widehat{\LLPO}(p)(k)
&\ni& \left\{\begin{array}{ll}
    0 & \iff(\forall n)\;\widehat{\LLPO}(p)\langle k,2n\rangle=0\\
    1 & \iff(\forall n)\;\widehat{\LLPO}(p)\langle k,2n+1\rangle=0
    \end{array}\right.\\
&=& \left\{\begin{array}{ll}
    0 & \iff(\forall n)(\forall m)\;p\langle\langle k,2n\rangle,2m\rangle=0\\
    1 & \iff(\forall n)(\forall m)\;p\langle\langle k,2n+1\rangle,2m\rangle=0
    \end{array}\right.\\
&=& \left\{\begin{array}{ll}
    0 & \iff(\forall\langle n,m\rangle)\;F(p)\langle k,2\langle n,m\rangle\rangle=0\\
    1 & \iff(\forall\langle n,m\rangle)\;F(p)\langle k,2\langle n,m\rangle+1\rangle=0\\
    \end{array}\right.\\
&\in& \widehat{\LLPO}\circ F(p)(k).        
\end{eqnarray*}
Thus, $\widehat{\LLPO}\circ\widehat{\LLPO}=\widehat{\LLPO}\circ F$, which proves
in particular that any realizer of $\widehat{\LLPO}$ computes a realizer 
of $\widehat{\LLPO}\circ\widehat{\LLPO}$, i.e.\ $\widehat{\LLPO}\circ\widehat{\LLPO}\leqSW\widehat{\LLPO}$.
\end{proof}

The analogous statement for $\widehat{\LPO}$ does not hold true. 

\begin{lemma}
\label{lem:LPO-completeness}
$\widehat{\LPO}\circ\widehat{\LPO}\nleqW\widehat{\LPO}$.
\end{lemma}
\begin{proof}
The function $\widehat{\LPO}=\C$
is known to be $\SO{2}$--complete and $\widehat{\LPO}\circ\widehat{\LPO}=\C^2$ is $\SO{3}$--complete
with respect to the effective Borel hierarchy \cite{Bra05}. That is $\C^2\nleqW\C$.
\end{proof}

Similarly as $\LPO$ translates Sierpi{\'n}ski space into the ordinary Boolean space,
we can consider $\LLPO$ as a translation of Kleene's ternary logic $K_3$ into ordinary
Boolean logic. Kleene's ternary logic uses the truth values $\IT:=\{0,1,\frac{1}{2}\}$,
where $\frac{1}{2}$ represents ``unknown''. Here we assume that $\IT$ is equipped with
the topology $\{\{0\},\{1\},\{0,1,\frac{1}{2}\}\}$ and with the canonical admissible
representation

\[\delta_\IT(p):=\left\{\begin{array}{ll}
  0 & \mbox{if $(\exists n)\;p(2n+1)\not=0$}\\
  1 & \mbox{if $(\exists n)\;p(2n)\not=0$}\\
  \frac{1}{2} & \mbox{if $(\forall n)\;p(n)=0$}
\end{array}\right.\]
with $\dom(\delta_\IT)=\dom(\LLPO)$. 

\begin{lemma}[Kleene's ternary logic]
The multi-valued map 
\[\mbox{$L:\IT\mto\{0,1\},L(0)=\{0\},L(1)=\{1\},L(\frac{1}{2})=\{0,1\}$}\]
is strongly Weihrauch equivalent to $\LLPO$.
\end{lemma}

The proof is obvious since $L$ and $\LLPO$ share the same realizations.
Now we define canonical extensions of Boolean functions from $\{0,1\}$ to $\IT$.

\begin{definition}\rm
For any Boolean function $f:\{0,1\}^n\to\{0,1\}$ we define the {\em ternary extension}
$f':\IT^n\to\IT$ by
\[f'(t_1,...,t_n):=L'f(L(t_1)\times...\times L(t_n))\]
where $L':\In2^{\{0,1\}}\to\IT,L'(\{0\})=0,L'(\{1\})=1, L'(\{0,1\})=\frac{1}{2}$.
\end{definition}

In this way, any ordinary Boolean operation can be transferred into its
counter part $f'$ in the strong version of Kleene's ternary logic $K_3$.
This holds in particular for the NAND operation $A\mid B$ and we explicitly
calculate the ternary truth table
that we obtain in this way:

\begin{table}[h]
\begin{tabular}{c|ccccccccc}
A       & 0 & 0 & 1 & 1 & 0 & 1 & $\frac{1}{2}$ & $\frac{1}{2}$ & $\frac{1}{2}$\\[0.1cm]
B       & 0 & 1 & 0 & 1 & $\frac{1}{2}$ & $\frac{1}{2}$ & 0 & 1 & $\frac{1}{2}$\\[0.1cm]\hline\\[-0.3cm]
$A\mid B$ & 1 & 1 & 1 & 0 & 1 & $\frac{1}{2}$ & 1 & $\frac{1}{2}$ & $\frac{1}{2}$
\end{tabular}
\end{table}

We prove that the NAND operation is computable on $\IT$.

\begin{lemma}
\label{lem:NAND}
The operation $(.\mid.):\IT\times\IT\to\IT,(A,B)\mapsto A\mid B$
is computable.
\end{lemma}
\begin{proof}
The function $N:\In\IN^\IN\to\IN^\IN$ with 
$\dom(N)=\langle\dom(\LLPO)\times\dom(\LLPO)\rangle$, defined by
\begin{eqnarray*}
N\langle 0^kb0^\IN,0^nc0^\IN\rangle&:=&\left\{\begin{array}{ll}
                                 0^{\max\{k,n\}+1}10^\IN & \mbox{if $k,n$ both even}\\
                                 0^{\min\{k,n\}+1}10^\IN & \mbox{if $\min\{k,n\}$ odd}
                          \end{array}\right.\\
N\langle 0^kb0^\IN,0^\IN\rangle&:=&N\langle 0^\IN,0^kb0^\IN\rangle := \left\{\begin{array}{ll}
                                 0^{k+1}10^\IN & \mbox{if $k$ odd}\\
                                 0^\IN         & \mbox{otherwise}
                          \end{array}\right.\\
N\langle 0^\IN,0^\IN\rangle       &:=& 0^\IN
\end{eqnarray*}
for all $k,n\in\IN$, $b,c\in\IN\setminus\{0\}$
is continuous and computable and it realizes the NAND operation with respect to $([\delta_\IT,\delta_\IT],\delta_\IT)$.
\end{proof}

As a corollary we obtain that all Boolean operations are computable
in the strong version of Kleene's ternary logic.

\begin{corollary}
\label{cor:ternary}
Let $f:\{0,1\}^n\to\{0,1\}$ be an arbitrary (Boolean) function.
Then the ternary counterpart $f':\IT^n\to\IT$ is computable.
\end{corollary}

The proof follows from the fact that any Boolean function $f$ can be realized
using substitutions of the NAND operation $\mid$ since the NAND operation
is complete. Using the extension of the NAND operation to $\IT$, the same
substitution yields the extension $f'$ of $f$ to $\IT$.
It is easy to see that this corollary even holds uniformly,
i.e.\ given a description of $f$ with respect to some standard representation
$[\delta_{\{0,1\}}^n\to\delta_{\{0,1\}}]$, we can find a description of $f'$
with respect to $[\delta_\IT^n\to\delta_\IT]$.

\begin{corollary}[Ternary extension]
\label{cor:uniform-ternary}
The operation
\[T:\{0,1\}^{\{0,1\}^n}\to\CC(\IT^n,\IT),\;f\mapsto f'\]
is $([\delta_{\{0,1\}}^n\to\delta_{\{0,1\}}],[\delta_\IT^n\to\delta_\IT])$--computable.
\end{corollary}

The next observation is that parallelized $\LLPO$ is upper semi-computable
as a set-valued operation. By $\KK_-(\{0,1\}^\IN)$ we denote the set of
all compact subsets of $\{0,1\}^\IN$ represented by the negative information
representation $\kappa_-$. A name of a compact set $K$ with respect to $\kappa_-$
is a list of all finite open rational covers of $K$. 

\begin{lemma}
\label{lem:compact-LLPO}
The function 
\[F:\In\IN^\IN\to\KK_-(\{0,1\}^\IN),p\mapsto\widehat{\LLPO}(p)\]
is computable.
\end{lemma}
\begin{proof}
Given a sequence $p=\langle p_0,p_1,p_2,...\rangle\in\dom(\widehat{\LLPO})$
one can start to inspect the sequences $p_0,p_1,...$ for an element different
from $0$. Whenever such an element is found in some $p_i$, then this provides
a piece of negative information on $F(p)$, since depending on whether the
non-zero position in $p_i$ occurs in an even or odd position, this implies
that $F(p)(i)$ is different from $0$ or $1$. In other words, in this moment
the negative information $\{0,1\}^i1$ or $\{0,1\}^i0$ can be enumerated.
This procedure describes the enumeration of a set $W\In\{0,1\}^*$ of words
such that $F(p)=\{0,1\}^\IN\setminus W\{0,1\}^\IN$. Such an enumeration
constitutes a $\psi_-$--name of $F(p)$, which can be translated into 
a $\kappa_-$--name of $F(p)$ since $\{0,1\}^\IN$ is computably compact
(see \cite{BP03}).
\end{proof}

As another auxiliary result we will use the following lemma
that guarantees that we can compute a modulus of uniform continuity
for computable functions on compact sets.

\begin{lemma}[Modulus of uniform continuity]
\label{lem:modulus}
For any computable function\linebreak $F:\In\{0,1\}^\IN\to\{0,1\}^\IN$
there is a computable multi-valued function
\[M:\In\KK_-(\{0,1\}^\IN)\mto\IN^\IN\]
such that $\dom(M)=\{K:K\In\dom(F)\}$ and any $m\in M(K)$ is 
a uniform modulus of continuity of $F$ on $K$, i.e.\
\[F(p[m(n)]\IN^\IN)\In F(p)[n]\IN^\IN\]
for all $p\in K$ and $n\in\IN$.
\end{lemma}
\begin{proof}
If $F:\In\{0,1\}^\IN\to\{0,1\}^\IN$ is computable, then there is
a computable monotone function $f:\{0,1\}^*\to\{0,1\}^*$ that
approximates $F$, i.e.\ such that $F(p)=\sup_{w\prefix p}f(w)$
for all $p\in\dom(F)$ (see \cite{Wei00}). Given $K\in\KK_-(\{0,1\}^\IN)$,
we can enumerate $W:=\{w\in\{0,1\}^*:w\IN^\IN\In\{0,1\}^\IN\setminus K\}$.
Any function $m:\IN\to\IN$ that satisfies
\[m(n)\geq\min\{k\in\IN:(\forall w\in\{0,1\}^k\setminus W)\;|f(w)|\geq n\}\]
is a modulus of uniform continuity of $F$ on $K$. Given $n\in\IN$ the set 
$M_n$ on the right hand side is non-empty, since $F$ is uniformly continuous 
on the compact set $K$ and although we might not be able to find 
$\min M_n$, we can certainly find some point $m(n)\in M_n$ by exhaustive search 
since $\{0,1\}^k$ is finite for any $k\in\IN$ and $W$ can be enumerated.
Thus, we can compute some modulus $m$ of uniform continuity of $F$.
\end{proof}

Using the NAND operation we can prove another interesting property of $\widehat{\LLPO}$,
namely that it has some quasi-continuity property although it is discontinuous 
and we will exploit this property for our main result in this section.
This result can also be interpreted as a completeness result for parallelized $\LLPO$.

\begin{theorem}[Completeness of parallelized $\LLPO$]
\label{thm:LLPO-swap}
For any comp\-ut\-able function $F:\In\{0,1\}^\IN\to\{0,1\}^\IN$ 
there exists a computable $G:\In\{0,1\}^\IN\to\{0,1\}^\IN$ 
such that 
\[F\circ\widehat{\LLPO}=\widehat{\LLPO}\circ G.\]
\end{theorem}
\begin{proof}
Let $F:\In\{0,1\}^\IN\to\{0,1\}^\IN$ be computable and let $p\in P:=\dom(F\circ\widehat{\LLPO})$.
Then we can compute $K_p=\widehat{\LLPO}(p)\in\KK_-(\{0,1\}^\IN)$ by Lemma~\ref{lem:compact-LLPO}.
Since $K_p\In\dom(F)$ it follows that $F$ is uniformly continuous on $K_p$
and given $p\in P$ we can compute a modulus of uniform continuity $m_p:\IN\to\IN$ of $F$ on $K_p$
by Lemma~\ref{lem:modulus}. We obtain
$F(q[m_p(n)]\IN^\IN)\In F(q)[n]\IN^\IN$
for all $n\in\IN$ and $q\in K_p$. 
Without loss of generality we can assume $m_p(n)\geq1$ for all $n\in\IN$.
Moreover, given $p\in P$ and $n\in\IN$, we can use
any machine for $F$ in order to compute Boolean functions 
$f_{p,n}:\{0,1\}^{m_p(n+1)}\to\{0,1\}$ such that 
\[F(q)(n)=f_{p,n}(q(0),...,q(m_p(n+1)-1))\]
for any $q\in K_p$ and for $n\in\IN$. By Corollary~\ref{cor:uniform-ternary} 
we can compute realizations $G_{p,n}:\In\{0,1\}^\IN\to\{0,1\}^\IN$
of the ternary extensions $f_{p,n}':\IT^{m_p(n+1)}\to\IT$ for any $p\in P$ and $n\in\IN$.
Now we define a computable function $G:\In\{0,1\}^\IN\to\{0,1\}^\IN$ by
\[G\langle p_0,p_1,...\rangle:=\langle G_{p,0}\langle p_0,...,p_{m_p(1)-1}\rangle,
                                       G_{p,1}\langle p_0,...,p_{m_p(2)-1}\rangle,...\rangle\]
for any $p=\langle p_0,p_1,p_2,...\rangle\in P$ and we obtain
\begin{eqnarray*}
&& \widehat{\LLPO}\circ G(p)\\
&=& \langle\LLPO\circ G_{p,0}\langle p_0,...,p_{m_p(1)-1}\rangle,\LLPO\circ G_{p,1}\langle p_0,...,p_{m_p(2)-1}\rangle,...\rangle\\
&=& \{\langle f_{p,0}(q(0),...,q(m_p(1)-1)),f_{p,1}(q(0),...,q(m_p(2)-1)),...\rangle:q\in K_p\}\\
&=& F(K_p)\\
&=& F\circ\widehat{\LLPO}(p)
\end{eqnarray*}  
as desired.
\end{proof}

If we combine the results proved so far, then we obtain that the multi-valued 
operations below $\widehat{\LLPO}$ are closed under composition. This has first
been observed in \cite{GM09}, where it was expressed in terms of Weak K\H{o}nig's Lemma
(see also Corollary~\ref{cor:hahn-banach}).

\begin{proposition}[Composition]
\label{prop:LLPO-comp}
Let $f:\In X\mto Y$ and $g:\In Y\mto Z$ be multi-valued operations on 
represented spaces. Then 
\[f\leqW\widehat{\LLPO}\mbox{ and }g\leqW\widehat{\LLPO}\TO g\circ f\leqW\widehat{\LLPO}.\]
The same holds true with respect to some oracle (i.e.\ we can replace Weih\-rauch reducibility
by its continuous counterpart in all occurrences here).
\end{proposition}
\begin{proof}
Let $f\leqW\widehat{\LLPO}$ and $g\leqW\widehat{\LLPO}$.
Since $\widehat{\LLPO}$ is a cylinder by Proposition~\ref{prop:LLPO-cylinder}  
there are computable functions $H,K,H',K':\In\IN^\IN\to\IN^\IN$ such that
$F=H'L'K'$ is a realizer of $f$ for any realizer $L'$ of $\widehat{\LLPO}$
and $G=HLK$ is a realizer of $g$ for any realizer $L$ of $\widehat{\LLPO}$.
This follows from Corollary~\ref{cor:cylinder-reduction}.
In particular, any realizer of $H\circ\widehat{\LLPO}\circ K\circ H'\circ\widehat{\LLPO}\circ K'$
is a realizer of $gf$. 
We can assume that $\dom(KH')\cup\range(KH')\In\{0,1\}^\IN$      
and hence by Theorem~\ref{thm:LLPO-swap}
there is a computable function $G$ such that $KH'\circ\widehat{\LLPO}=\widehat{\LLPO}\circ G$.
Let $F$ be the computable function according to Lemma~\ref{lem:LLPO-double} such that
$\widehat{\LLPO}\circ\widehat{\LLPO}=\widehat{\LLPO}\circ F$. Then with $K'':=FGK'$ we obtain
\[H\circ\widehat{\LLPO}\circ K\circ H'\circ\widehat{\LLPO}\circ K'=H\circ\widehat{\LLPO}\circ K'',\]
which implies that for any realizer $L$ of $\LLPO$ the function $HLK''$ is a realizer
of $gf$. Hence, $gf\leqW\widehat{\LLPO}$.
In the presence of some oracle the reasoning is analogously,
only the functions $H,K,H',K'$ have to be replaced by continuous ones.
\end{proof}

We believe that this result justifies to give a new name to the operations below
$\widehat{\LLPO}$. 

\begin{definition}[Weakly computable]\rm
A function $f:\In X\mto Y$ on represented spaces $X,Y$ is called {\em weakly computable},
if $f\leqW\widehat{\LLPO}$. Similarly, such a function is called {\em weakly continuous},
if $f\leqW\widehat{\LLPO}$ holds with respect to some oracle.
\end{definition}

One main goal of this section is to prove the following theorem
on the omniscience principles. This theorem completely characterizes the relation 
of the omniscience principles and their parallelizations with respect to Weihrauch 
reducibility.

\begin{theorem}[Omniscience principles]
\label{thm:omniscience}
We obtain 
\[\LLPO\lW\LPO\nW\widehat{\LLPO}\lW\widehat{\LPO}.\]
All negative results also hold true with respect to some arbitrary oracle.
\end{theorem}
\begin{proof}
Firstly, it is easy to see that $\LLPO\leqSW\LPO$.
We define a computable function $K:\IN^\IN\to\IN^\IN$ by
$K(p)(n):=1\dmin p(2n)$ and we obtain
\begin{itemize}
\item $\LPO(K(p))=0\TO(\exists n)\;p(2n)\not=0\TO 1\in\LLPO(p)$,
\item $\LPO(K(p))=1\TO(\forall n)\;p(2n)=0\TO0\in\LLPO(p)$.
\end{itemize}
With a simple negating function $H$ this shows that $\LLPO\leqSW\LPO$.

Since parallelization is a closure operator by Proposition~\ref{prop:parallelization},
it follows that $\widehat{\LLPO}\leqW\widehat{\LPO}$. 
Let us now assume that $\widehat{\LPO}\leqW\widehat{\LLPO}$ with respect to some oracle.
Then 
\[\widehat{\LPO}\circ\widehat{\LPO}\leqW\widehat{\LLPO}\leqW\widehat{\LPO}\] 
would follow with respect to some oracle since by Proposition~\ref{prop:LLPO-comp} 
all operations below $\widehat{\LLPO}$ are closed under composition.
However, the above is a contradiction to Lemma~\ref{lem:LPO-completeness}.
Thus $\widehat{\LPO}\nleqW\widehat{\LLPO}$. 
This also implies $\LPO\nleqW\widehat{\LLPO}$ and $\LPO\nleqW\LLPO$ since parallelization 
is a closure operator by Proposition~\ref{prop:parallelization}.
Finally, $\widehat{\LLPO}\nleqW\LPO$ follows from the Mind Change Property proved in \cite{BG09b}, 
since $\LPO$ can be computed with at most one mind change by Proposition~\ref{prop:LPO-LLPO-mind-change}, 
whereas it is easy to see that $\widehat{\LLPO}$ cannot be computed with one mind change.
All the results hold true with respect to some oracle.
\end{proof}

Note that the proof even shows the strong reduction $\LLPO\leqSW\LPO$.
A different direct proof of $\LPO\nleqW\LLPO$ is presented in Theorem~4.2 in \cite{Wei92c}.

Since any discontinuous single-valued function is already above $\LPO$,
it is clear that no such single-valued function can be below $\widehat{\LLPO}$.
In other words, the parallel Weihrauch degree of $\LLPO$ has no single-valued member.
In particular, this means that multi-valuedness does not appear
accidentally in our theory, but in some sense it is unavoidable.
Indeed we will prove in Corollary~\ref{cor:single-valued-weakly-computable}
that any single-valued weakly computable function is already computable
in the ordinary sense.

\section{Compact choice and Weak K\H{o}nig's Lemma}
\label{sec:WKL}

In this section we will prove that the parallel version of $\LLPO$ is equivalent
to Weak K\H{o}nig's Lemma. We first formalize Weak K\H{o}nig's Lemma
for this purpose. We recall that a {\em binary tree} is a subset $T\In\{0,1\}^*$ that
is closed under the prefix relation, i.e.\ if $w\in T$ and $v\prefix w$, then $v\in T$. 
We use some standard bijective enumeration $(w_n)_{n\in\IN}$ of all the binary words.
By $\Tr$ we denote the set of all binary trees and we use a representation $\delta_\Tr$ of $\Tr$
that is defined by
\[\delta_{\Tr}(p)=T:\iff\chi_T(w_n)=p(n),\]
where $\chi_T:\{0,1\}^*\to\{0,1\}$ denotes the characteristic function of the binary tree $T$.
The classical statement of K\H{o}nig's Lemma is that any infinite binary tree has 
an infinite path. An infinite path of $T$ is a sequence $p\in\{0,1\}^\IN$, such that 
$p[n]\in T$ for all $n\in\IN$. 
By $[T]$ the set of infinite paths of $T$ is denoted.
Now we can formalize Weak K\H{o}nig's Lemma as follows.

\begin{definition}[Weak K\H{o}nig's Lemma]\rm
We define a multi-valued operation 
\[\WKL:\In\Tr\mto\{0,1\}^\IN,T\mapsto[T]\]
with $\dom(\WKL)=\{T\In\{0,1\}^*:T$ is an infinite binary tree$\}$.
\end{definition}

Weak K\H{o}nig's Lemma has already been studied in this form in \cite{GM09}. Our main result
here is that the parallel version of $\LLPO$ is strongly equivalent to Weak K\H{o}nig's Lemma.
For the proof we use Weak K\H{o}nig's Lemma itself.

\begin{theorem}[Weak K\H{o}nig's Lemma]
\label{thm:WKL}
$\WKL\equivSW\widehat{\LLPO}$.
\end{theorem}
\begin{proof}
We first prove $\WKL\leqSW\widehat{\LLPO}$. Given an infinite binary tree $T$,
we want to find an infinite path $p\in\{0,1\}^\IN$ in $T$.
For $n\in\IN$ and $i\in\{0,1\}$ we consider the following sets
\[P_{n,i}:=\{w\in\{0,1\}^*:(\forall v\in\{0,1\}^n\cap T)(v\not\prefix wi\mbox{ and }wi\not\prefix v)\}.\]
Intuitively, $P_{n,i}$ is the set of those nodes $w$ whose extension $wi$ is not on a path of $T$ of length
$n$.
Now we define a partial function $m:\In\{0,1\}^*\to\IN$ for any word $w\in\{0,1\}^*$
by
\[m(w):=\min\{n\in\IN:w\in P_{n,0}\cup P_{n,1}\}.\]
Intuitively, $m(w)$ is the shortest length such that $w0$ or $w1$ is not on a path of $T$ of that length.
Now we construct a sequence $q_w\in\{0,1\}^\IN$ for any word $w\in\{0,1\}^*$ 
as follows:
\[q_w:=\left\{\begin{array}{ll}
0^{2n}10^\IN   & \mbox{if $n=m(w)$ exists and $w\in P_{n,0}\setminus P_{n,1}$}\\
0^{2n+1}10^\IN & \mbox{if $n=m(w)$ exists and $w\in P_{n,1}\setminus P_{n,0}$}\\
0^\IN          & \mbox{otherwise}
\end{array}\right..\]
Given the binary tree $T$, we can actually compute the sequence $\langle q_{w_0},q_{w_1},...\rangle$.
Moreover, we obtain for all $w\in T$ and $i\in\{0,1\}$
\begin{eqnarray*}
       i\in\LLPO(q_w)
&\TO& (\forall n)\;q_w(2n+i)=0\\
&\TO& (\forall n)(\exists v\in\{0,1\}^n\cap T)(v\prefix wi\mbox{ or }wi\prefix v)\\
&\TO& (\exists p\in[T])\;wi\prefix p.
\end{eqnarray*}
By Weak K\H{o}nig's Lemma we obtain an infinite path $p\in[T]$ inductively,
by selecting $p(0)=i$ such that $i\in\LLPO(q_\varepsilon)$ for the empty word $\varepsilon$ 
and given a prefix $p[n]$
we choose $p(n)=i$ such that $i\in\LLPO(q_{p[n]})$.
Given a realizer of $\widehat{\LLPO}$, we can determine some
$r\in\widehat{\LLPO}\langle q_{w_0},q_{w_1},...\rangle$ and using
the inductive method that we just described, we can actually
compute an infinite path with the help of $r$.
Altogether, this shows $\WKL\leqSW\widehat{\LLPO}$.

Now we prove $\widehat{\LLPO}\leqSW\WKL$.
Given $p=\langle p_0,p_1,p_2,...\rangle\in\dom(\widehat{\LLPO})$, we want to
find some $r\in\widehat{\LLPO}(p)$.
For this purpose we construct a binary tree $T$ that can be computed from $p$:
\[T:=\{i_0...i_{n-1}\in\{0,1\}^{n}:(\forall m,k<n)\;p_m(2k+i_m)=0\}.\]
An infinite path $r$ of this binary tree $T$ satisfies the desired condition
since 
\[r(n)=i\iff(\forall k)\;p_n(2k+i)=0\iff i\in\LLPO(p_n)\]
for all $n\in\IN$ and $i\in\{0,1\}$. Thus, given a realizer of $\WKL$,
we can compute the desired $r$. Altogether, this proves $\widehat{\LLPO}\leqSW\WKL$.
\end{proof}

In \cite{GM09} it has been proved that the Hahn-Banach Theorem $\HBT$ has
the same Weihrauch degree as $\WKL$ and hence the same Weihrauch degree as $\widehat{\LLPO}$.
We formulate this as a corollary without exactly specifying $\HBT$ (the reader
is referred to \cite{GM09} for details).

\begin{corollary}
\label{cor:hahn-banach}
$\HBT\equivW\WKL\equivW\widehat{\LLPO}$.
\end{corollary}

Another equivalence that has been proved in \cite{GM09} is that all the 
aforementioned theorems are equivalent to compact choice in rich spaces. 
We will use this observation and we adapt the formulation to our context.

\begin{definition}[Compact choice]\rm
Let $X$ be a computable metric space.
The multi-valued operation 
\[\C_{\KK(X)}:\In\KK_-(X)\mto X,A\mapsto A\]
with $\dom(\C_{\KK(X)}):=\{A\In X:A\not=\emptyset$ compact$\}$ is called {\em compact choice} of $X$.
\end{definition}

Here $\KK_-(X)$ denotes the set of compact subsets of $X$, which is equipped
with the negative information representation $\kappa_-$ 
(here a name of a compact set $K$ is a list of all finite open rational covers of $K$, 
see \cite{BP03} for details).
In some sense, $\WKL$ is compact choice for the Cantor space $\{0,1\}^\IN$ and,
in fact, in \cite{GM09} it has been proved that compact choice for a large class
of computable metric space is equivalent to $\C_{\KK(\{0,1\}^\IN)}\equivW\WKL$.
Using this result we prove a slightly different result here adapted to our operations.

\begin{theorem}[Compact choice]
\label{thm:CK}
Let $X$ be a computable metric space. Then $\C_{\KK(X)}\leqSW\widehat{\LLPO}$.
If $X$ is rich, i.e.\ if there is a computable embedding $\iota:\{0,1\}^\IN\into X$,
then $\C_{\KK(X)}\equivSW\widehat{\LLPO}$.
\end{theorem}
\begin{proof}
In Theorem~8.3 of \cite{GM09} the operation 
\[{\mathrm{Sel}}_{\KK(X)}:\In\KK_-(X)\times\AA_-(X)\mto X,(K,A)\mapsto A\]
with $\dom({\mathrm{Sel}}_{\KK(X)})=\{(K,A):\emptyset\not= A\In K\}$
has been considered and it has been proved that ${\mathrm{Sel}}_{\KK(X)}\leqW\WKL$.
Here $\AA_-(X)$ is the set of closed
subsets of $X$ represented with the negative information representation
$\psi_-$ (see \cite{BP03}). 
In fact, in \cite{GM09} a stronger representation $\kappa$
of the set $\KK(X)$ of compact subsets of $X$ has been used, but a careful inspection of the proof shows
that only $\kappa_-$--information has been exploited. 
Since the injection
$\KK_-(X)\into\AA_-(X)$ is $(\kappa_-,\psi_-)$--computable (see Theorem~4.8 in \cite{BP03})
and $\C_{\KK(X)}(K)={\mathrm{Sel}}_{\KK(X)}(K,K)$ we obtain
\[\C_{\KK(X)}\leqW{\mathrm{Sel}}_{\KK(X)}\leqW\WKL\equivW\widehat{\LLPO}\]
according to Theorem~\ref{thm:WKL}. Since $\widehat{\LLPO}$ is a cylinder,
this implies the strong reducibility $\C_{\KK(X)}\leqSW\widehat{\LLPO}$.

Now let us assume that $X$ is additionally rich, i.e.\ 
there is a computable embedding $\iota:\{0,1\}^\IN\into X$.
It is clear that $\KK_-(\{0,1\}^\IN)\to\KK_-(X),A\mapsto\iota(A)$ is $(\kappa_-,\kappa_-)$--computable
(see \cite{Wei03}). 
In the proof of Theorem~8.3 of \cite{GM09} it was already shown that $\WKL\leqSW\C_{\KK(\{0,1\}^\IN)}$.
Similarly as in that proof we obtain $\C_{\KK(\{0,1\}^\IN)}(A)=\iota^{-1}\C_{\KK(X)}(\iota(A))$.
Since the partial inverse $\iota^{-1}:\In X\to\{0,1\}^\IN$ is computable (by Corollary~6.5 in \cite{Bra08}),
we can conclude with Theorem~\ref{thm:WKL}
\[\widehat{\LLPO}\equivSW\WKL\leqSW\C_{\KK(\{0,1\}^\IN)}\leqSW\C_{\KK(X)}.\]
Thus $\widehat{\LLPO}\equivW\C_{\KK(X)}$ for rich computable metric spaces $X$.
\end{proof}

The characterization of $\widehat{\LLPO}$ as compact choice allows us to prove
a characterization of weakly computable operations. As a preparation we need
to prove the following lemma. It shows that a compact union of compact sets
can be computed and it generalizes an observation made in \cite{Bra01a}.

\begin{lemma}[Compact union]
\label{lem:compact-union}
Let $Y,Z$ be computable metric spaces and $X$ some represented space. 
Let $f:\In X\to\KK_-(Z)$ and $g:\In Z\to\KK_-(Y)$ be computable.
Then
\[h:\In X\to\KK_-(Y),x\mapsto\bigcup_{z\in f(x)}g(z)\]
is computable, where $\dom(h):=\{x\in X:f(x)\In\dom(g)\}$.
In particular, $h$ is well-defined.
\end{lemma}
\begin{proof}
Firstly, $h$ is well-defined, for instance by Corollary~9.6 in \cite{Mic51}.
Since $g:\In Z\to\KK_-(Y)$ is $(\delta_Z,\kappa_-)$--computable, it follows
that 
\[G:\OO(Y)\to\OO(\dom(g)),U\mapsto\{z\in Z:g(z)\In U\}\]
is $(\vartheta,\vartheta)$--computable. Here $\OO(Y)$ denotes the set of open
subsets of $Y$ represented by $\vartheta$, where a name of an open set $U\In Y$
is a list of open balls $B(x_i,r_i)$ with $U=\bigcup_{i=0}^\infty B(x_i,r_i)$ 
with centers $x_i$ in the dense subset of $Y$ and rational radii $r_i$.
Analogously, $\OO(\dom(g))$ denotes the open subsets of $\dom(g)$ with respect
to the subspace topology of $Z$ and a name of an open subset $V\In\dom(g)$ 
is a list of rational balls of $Z$ with $V=\dom(g)\cap\bigcup_{i=0}^\infty B(x_i,r_i)$.
We obtain
\[h(x)=\bigcup_{z\in f(x)}g(z)\In U\iff(\forall z\in f(x))\;g(z)\In U\iff f(x)\In G(U)\]
for all open $U\In Y$ and $x\in\dom(f)$ such that $f(x)\In\dom(g)$.
Since $f$ is $(\delta_X,\kappa_-)$--computable it follows that the right-hand side condition
is c.e.\ in $x$ and $U$. Hence it follows that the left-hand side
condition has the same property and $h$ is $(\delta_X,\kappa_-)$--computable.
\end{proof}

Now we are prepared to prove the characterization of weakly computable operations.
We say that a function $s:\In X\to\KK_-(Y)$ is a {\em selector} of a function
$f:\In X\mto Y$, if $\dom(s)=\dom(f)$ and $s(x)\In f(x)$ for all $x\in\dom(f)$.
Continuous functions $s:\In X\to\KK_-(Y)$ 
are also called {\em upper semi-continuous}.

\begin{theorem}[Selection]
\label{thm:selection}
Let $X$ be a represented space and let $Y$ be a computable metric space.
A function $f:\In X\mto Y$ is weakly computable if and only if $f$ admits a computable
selector $s:\In X\to\KK_-(Y)$.
\end{theorem}
\begin{proof}
Let $s:\In X\to\KK_-(Y)$ be a computable selector of $f$. Then any realizer of 
$\C_{\KK(Y)}\circ s$ is also a realizer of $f$ and hence by Theorem~\ref{thm:CK}
we obtain
\[f\leqW\C_{\KK(Y)}\circ s\leqW\C_{\KK(Y)}\leqW\widehat{\LLPO}.\]
Thus, $f$ is weakly computable.

Let $\delta_X$ be the representation of $X$ and let $\delta_Y$ be the Cauchy representation
of $Y$.  
Now let $f:\In X\mto Y$ be weakly computable. Then by Theorem~\ref{thm:CK} we obtain
$f\leqW\widehat{\LLPO}\equivW\C_{\KK(\{0,1\}^\IN)}$
and hence there are computable functions $H$ and $K$ such that
$K\langle\id,GH\rangle$ is a $(\delta_X,\delta_Y)$--realizer of $f$ for any 
$(\kappa_-,\id|_{\{0,1\}^\IN})$--realizer $G$ of $\C_{\KK(\{0,1\}^\IN)}:\In\KK_-(\{0,1\}^\IN)\mto\{0,1\}^\IN$. 
Without loss of generality we can assume that $\delta_X$ is a representation
of $X$ with compact fibers $\delta_X^{-1}\{x\}$ and such that 
$X\to\KK_-(\{0,1\}^\IN),x\mapsto\delta_X^{-1}\{x\}$ is $(\delta_X,\kappa_-)$--computable.
An example of such a representation is Schr\"oder's representation (see Theorem~4.4 in \cite{Wei03}).
Since $H$ is computable, it follows with Lemma~\ref{lem:compact-union} that
\[h:\In X\to\KK_-(\{0,1\}^\IN),x\mapsto\bigcup_{p\in\delta_X^{-1}\{x\}}\kappa_-H(p)\]
with $\dom(h)=\dom(f)$ is $(\delta_X,\kappa_-)$--computable since $\delta_X^{-1}\{x\}\In\dom(H)$ for all $x\in\dom(f)$.
Since $K$ is computable and the Cartesian product operation 
$\times:\KK_-(\{0,1\}^\IN)\times\KK_-(\{0,1\}^\IN)\to\KK_-(\{0,1\}^\IN),(A,B)\mapsto \langle A\times B\rangle$
is $(\kappa_-,\kappa_-,\kappa_-)$--computable,
it follows from Theorem~3.3 in \cite{Wei03} that
\[s:\In X\to\KK_-(Y),x\mapsto\delta_YK\langle \delta_X^{-1}\{x\}\times h(x)\rangle\]
is $(\delta_X,\kappa_-)$--computable and $s$ is a selector of $f$.
\end{proof}

It is known that for computable metric spaces $(Y,\delta_Y)$ the singleton operation 
$Y\to\KK_-(Y),y\mapsto\{y\}$ that maps a point to the corresponding singleton set 
is $(\delta_Y,\kappa_-)$--computable and it admits a $(\kappa_-,\delta_Y)$--computable
right inverse (see for instance Lemma~6.4 in \cite{Bra08}). 
Thus we obtain the following corollary of the Selection Theorem~\ref{thm:selection}.

\begin{corollary}[Weakly computability]
\label{cor:single-valued-weakly-computable}
Let $X$ be a represented space and $Y$ a computable metric space.
Any weakly computable single-valued operation $f:\In X\to Y$ is computable.
\end{corollary}

Similarly, it follows that any weakly continuous single-valued function
is already continuous in the ordinary sense.
We close this section with a characterization of $\LLPO$ that is well-known
from constructive analysis (see for instance \cite{BR87a}).

\begin{lemma}[Real $\LLPO$]
We define 
\[\LLPO_\IR:\IR\mto\IN,\LLPO_\IR(x)\ni\left\{\begin{array}{ll}
       0 & \mbox{if $x\leq 0$}\\
       1 & \mbox{if $x\geq 0$}
       \end{array}\right.
\]
Then $\LLPO_\IR\equivSW\LLPO$.
\end{lemma}
\begin{proof}
We prove $\LLPO\leqSW\LLPO_\IR$. Given a sequence $p\in\dom(\LLPO)$
we let 
$k:=\min\{i\in\IN:p(2i)\not=0\mbox{ or }p(2i+1)\not=0\}$
and we define
\[x:=\left\{\begin{array}{ll}
2^{-k}  & \mbox{if $k$ exists and $p(2k)\not=0$}\\
-2^{-k} & \mbox{if $k$ exists and $p(2k+1)\not=0$}\\
0       & \mbox{if $k$ does not exist}
\end{array}\right.
\]
Given $p$ we can compute $x$ and $\LLPO_\IR(x)=\LLPO(p)$.
This proves the desired reduction $\LLPO\leqSW\LLPO_\IR$.

We now prove $\LLPO_\IR\leqSW\LLPO$. Given $x\in\IR$ 
compute a sequence $p\in\IN^\IN$ such that
\[p=\left\{\begin{array}{ll}
0^{2i}10^\IN  & \mbox{if $x<0$}\\
0^{2j+1}10^\IN & \mbox{if $x>0$}\\
0^\IN          & \mbox{if $x=0$}
\end{array}\right.\]
for some $i,j\in\IN$. 
Such $i,j$ can just be determined by an exhaustive search.
In stages $n=0,1,2,...$ one simultaneously tries to verify $x<0$ and $x>0$
and depending on which question is answered first, one uses the corresponding
stage as $i$ or $j$. As long as no answer is available, the algorithm
just produces zeros.
Then $\LLPO(p)=\LLPO_\IR(x)$ and hence
$\LLPO_\IR\leqSW\LLPO$.
\end{proof}


\section{Conclusions}

In this paper we have studied Weihrauch reducibility of multi-valued
functions on represented spaces. Among other things, we have proved
that Weihrauch degrees form a lower semi-lattice with the direct sum
operation as greatest lower bound operation. Moreover, we have studied
parallelization as closure operator and we have shown that the parallelized
Weihrauch degrees even form a lattice with the product as least upper
bound operation. 
The Medvedev lattice and the upper semi-lattice of
Turing degrees can be embedded into the parallelized Weihrauch lattice.
Moreover, we have proved that the parallelized versions $\widehat{\LPO}$ and
$\widehat{\LLPO}$ of the limited principle of omniscience and the lesser
limited principle of omniscience, respectively, play a crucial role in our
lattice. While $\widehat{\LPO}$ is complete for the class of limit computable
operations, we have shown that $\widehat{\LLPO}$ can be used to define
a meaningful class of weakly computable operations that is closed under
composition. Single-valued weakly computable operations are already computable
in the ordinary sense. This fact could be related to conservativeness properties
of $\WKL_0$ in reverse mathematics \cite{Sim99,STY02} and to known uniqueness properties in 
constructive mathematics \cite{Koh93,Sch06c,Sch07d,Ish07}. 

In a forthcoming paper \cite{BG09b} we discuss the classification of the 
Weihrauch degree of many theorems
from analysis, such as the Intermediate Value Theorem, the Baire Category Theorem,
the Banach Inverse Mapping Theorem and many others.
It turns out that certain choice principles are crucial cornerstones for
that classification and we believe that our classification sheds new
light on the computational properties of these theorems.
In particular, our classification seems to be in a well-defined sense finer 
than other known classifications in constructive and reverse mathematics.

\bibliographystyle{asl}
\bibliography{../../bibliography/new/lit,local}

\end{document}